\documentclass[11pt,a4paper,reqno]{amsart}  
\usepackage{ifthen,amssymb,mathrsfs,a4wide,color,mathtools}
\mathtoolsset{showonlyrefs}
\let\spacecal=\mathscr
\usepackage[all]{xypic}
\usepackage[backref=page]{hyperref}
\DeclareMathAlphabet{\mathpzc}{OT1}{pzc}{m}{it}
\hypersetup{
 colorlinks,
 citecolor=green,
 linkcolor=blue,
 urlcolor=Blue}
\makeatletter
\@namedef{subjclassname@2020}{%
  \textup{2020} Mathematics Subject Classification}
\makeatother
\definecolor{maroon}{rgb}{0.5, 0.0, 0.0}
\definecolor{brick}{rgb}{0.79, 0.25, 0.33}

\newcommand{\blue}[1]{\textcolor{blue}{#1}}

\newcommand{\marginnote}[1]{\relax}
\let\tempone\itemize
\let\temptwo\enditemize
\let\temponenum\enumerate
\let\temptwonum\endenumerate
\renewenvironment{itemize}{\tempone\addtolength{\itemsep}{0.2\baselineskip}}{\temptwo}
\renewenvironment{enumerate}{\temponenum\addtolength{\itemsep}{0.2\baselineskip}}{\temptwonum}
\let\tempsec\section
\renewcommand{\section}{\par\medskip\tempsec}

\newcommand{\qee}{\parbox{5pt}{\hfill}\hfill $\triangle$}

\newenvironment{rem}{\begin{remqee}}{\qee\end{remqee}}
\newtheorem{thm}{Theorem}[section] 
\newtheorem{corol}[thm]{Corollary}
\newtheorem{lemma}[thm]{Lemma} 
\newtheorem{prop}[thm]{Proposition}
\newtheorem{defin}[thm]{Definition}
\theoremstyle{definition}

\theoremstyle{remark}
\newtheorem{remqee}[thm]{Remark}

\newtheorem{example}[thm]{Example}
\newenvironment{remark}{\begin{remqee}}{\qee\end{remqee}}
\numberwithin{equation}{section}

\newcommand\Id{\operatorname{Id}}
\newcommand\Ima{\operatorname{Im}}

\newcommand\Spec{\operatorname{Spec}}
\newcommand\Proj{\operatorname{Proj}}
\newcommand\Ext{\operatorname{Ext}}
\newcommand\Hom{\operatorname{Hom}}

\newcommand\iso{\kern.35em{\raise3pt\hbox{$\sim$}\kern-1.1em\to}\kern.3em}

\newcommand\rest[2]{#1_{\vert #2}}

\newcommand\F{{\mathcal F}}
\newcommand\Oc{{\mathcal O}}

\newcommand\M{{\mathcal M}}
\newcommand\R{{\mathbb R}}
\newcommand\Z{{\mathbb Z}}

\newcommand\C{{\mathbb C}}

\newcommand\K{{\mathbb K}}

\newcommand\U{{\mathcal U}}
\newcommand\Dcal{{\mathcal D}}
\newcommand\Nc{{\mathcal N}}

\newcommand\Lcl{{\mathcal L}}
\newcommand\Ec{{\mathcal E}}
\newcommand\Ac{{\mathcal A}}

\newcommand\Ps{{\mathbb P}}

\newcommand\Xcal{{\spacecal X}}
\newcommand\Ycal{{\spacecal Y}}
\newcommand\Sc{{\spacecal S}}
\newcommand\Tc{{\spacecal T}}
\newcommand\Zc{{\spacecal Z}}

\newcommand\Wc{{\spacecal W}}

\newcommand\Hc{{\mathcal H}}
\newcommand\Ucal{{\spacecal U}}
\newcommand\Rcal{{\spacecal R}}

\newcommand\Ic{{\mathcal I}}
\newcommand\Jc{{\mathcal J}}
\newcommand\Cc{{\mathcal C}}

\newcommand\bR{{\mathbf R}}
\newcommand\bL{{\mathbf L}}

\newcommand\Ber{{{\mathcal B}er}}
\newcommand\Homsh{{{\mathcal H}om}}

\newcommand\Xf{{\mathfrak X}}

\newcommand\Sf{{\mathfrak S}}

\newcommand\mf{{\mathfrak m}}

\newcommand\af{{\mathfrak a}}
\newcommand\bfr{{\mathfrak b}}

\newcommand\nf{{\mathfrak n}}
\newcommand\Mf{{\mathfrak M}}
\newcommand\As{{\mathbb A}}
\newcommand\Bs{{\mathbb B}}

\newcommand\SSpec{\mathbb{S}\mathrm{pec}\,}
\newcommand\SProj{\mathbb{P}\mathrm{roj}\,}

\newcommand\Hs{{\mathbb H}}

\newcommand{\rra}{\rightrightarrows}

\newcommand\Extsh{{{\mathcal E}xt}}

\newcommand{\bos}{{\operatorname{bos}}}

\newcommand{\ord}{\operatorname{ord}}
\newcommand{\divi}{\operatorname{div}}
\newcommand{\ber}{\operatorname{ber}}

\begin{document}

\renewcommand{\leftmark}{\sc U.Bruzzo, D. Hern\'andez Ruip\'erez, Y. I. Manin}
 
\title[Supercycles, stable supermaps and SUSY Nori Motives]{SUPERCYCLES, STABLE SUPERMAPS AND SUSY NORI MOTIVES}
\author{Ugo Bruzzo, Daniel Hern\'andez Ruip\'erez, Yuri I.\, Manin}
 
 \thanks{U.B.: Research partly supported  by the  Brazilian CNPq ``Bolsa de Produtividade em Pesquisa'' 313333/2020, by GNSAGA-INdAM, and by the PRIN project ``Moduli Theory and Birational Classification.'' DHR: Research partly supported by research projects  ``Espacios finitos y functores integrales'', MTM2017-86042-P (Ministerio de Econom\'{\i}a, Industria y Competitividad), and  ``STAMGAD, SA106G19'' (Junta de Castilla y Le\'on).}
 
\date{Revised \today}

\begin{abstract} 
We define stable supercurves and stable supermaps, and based on these notions we develop a theory of Nori motives for the category of stable supermaps of SUSY curves with punctures. This will require several preliminary constructions, including the development of a basic theory of supercycles.
\end{abstract}

\maketitle



\tableofcontents

\section{Introduction}
This is a revision of the paper \cite{BrHRMan} published in Algebraic Geometry and Physics, which incorporates the corrections we made in the  proofs, and an Erratum which is going to be published in the same journal. The corrections amount to some typos, and a partial reformulation of Remark \ref{rem:rifatto}, which however does not affect the rest of the paper.

 Motives were introduced by Grothendieck in a letter to Serre in 1964,  with the aim to introduce  an abelian category attached to algebraic varieties through whose objects all cohomology theories factorize. Pure motives are  attached to projective smooth varieties and mixed motives are attached  to all kinds of varieties, regardless of smoothness. In spite of the enormous progress made in the field, the theory is still quite conjectural. There are various constructions of   categories of motives, but their expected properties are often not proved, or depend on the proof of the standard conjectures. There are three candidates for an abelian category of mixed motives: the absolute Hodge motives of Deligne and Jansen,   the categories of Ayoub, and that of Nori (see the introduction of \cite{HuM-St17}). There are also candidates of triangulated categories of motives, due to Hanamura, Levine and Voevodsky, that one can somehow   think as the derived categories of the ``true'' category of motives.

 As remarked in \cite{HuM-St17} regarding the cohomology theories ``similar''  to singular cohomology (in a suitable sense), they all take values in rigid tensor categories. We then expect   the conjectural abelian category of motives to be a Tannakian category such that singular cohomology gives a faithful exact functor to the category of rational vector spaces. Nori started from there, and his category of motives is universal for all the cohomology theories compatible with singular cohomology, and has a natural functor from the categories of pairs $(X,Y)$ where is $Y$ is a closed subvariety of $X$ and it is compatible with the exact sequences given by triples $(Z,Y,X)$.
The basic combinatorial objects in the theory of Nori motives are categories of diagrams \cite{MaMar20}. In \cite{BoMa08} and \cite{Ma99} a slightly modified formalism refers to graphs, and we shall indeed use the formalism of \cite{BoMa08, Ma99} rather than \cite{MaMar20}.

An object of such a category, a graph $\tau$, is a family of structures  $(F_\tau,V_\tau,\partial_\tau,j_\tau)$, just sets or structured sets and maps between them in the simplest cases. Elements of $F_\tau$, resp.\,$V_\tau$, are called \emph{flags}, resp.\,\emph{vertices} of $\tau$. The map $\partial_\tau\colon F_\tau \to V_\tau$ associates to each flag a vertex, called its \emph{boundary}. The map $j_\tau\colon F_\tau \to F_\tau$ must satisfy the condition $j_\tau^2 = \Id$, the identity map of $F_\tau$ to itself. Pairs of flags $(f,j_\tau f)$ are called \emph{edges}, connecting boundary vertices of these two flags.
A morphism of graphs $\sigma \to \tau$ consists of two maps $F_\sigma \to F_\tau$, $V_\sigma\to V_\tau$ compatible with the $\partial$'s and $j$'s. 
To each (small) category $\Cc$ one can associate its graph $D(\Cc)$, whose vertices are objects of $\Cc$ and flags are morphisms $\ast\to X$ and $X \to\ast$, where $\ast$ runs over all objects of $\Cc$. According to \cite[Def.\, 0.1.1]{MaMar20}, edges, oriented from $X$ to $Y$, are represented by diagrams $X \to Z \to Y$, that is, by decompositions of morphisms from $X$ to $Y$, presented as a  composition  of two morphisms.
Functors between two categories produce morphisms between their diagrams.

After the  
the introduction of super analogues of moduli spaces of stable curves (see \cite{Del87,FKP20, BrHRPo20,BrHR21}), it became clear that super-analogues of motives   can and should be developed.  In particular, to embed the notion of supersymmetry  into the theory, one may like to consider   a  {\em category of SUSY motives},  which should contain at least motives of moduli spaces/stacks of stable supersymmetric (SUSY) curves. The construction of SUSY motives is related to the graphs associated with stable supermaps, which should be a generalization of SUSY graphs. The latter have been recently studied in~\cite{KeMaWu22}.

These notes introduce super-analogues of Nori motives (for the classical case, see \cite{HuM-St17}). In a further paper we plan to introduce and study the analogues of Beilinson motives, whose classical case is developed in \cite{CiDe19}.
The introduction of SUSY Nori motives requires a previous study of stable supercurves and supermaps; for that reason, a sizeable part of this paper is devoted to these topics.

In Section \ref{s:derham} we recall and develop some basic theory of de Rham and Hodge cohomology in the super setting, including some elements of the cohomology of integral superforms.  This material is included here  by   sake of completeness and   is not used anywhere else in the paper.

In Section \ref{s:supercycles} we define  supercycles, together with rational equivalence, flat pullbacks and proper push-forwards, while in Section  \ref{supercurves}
we give a definition of pre-stable and stable supersymmetric (SUSY) supercurves. In Section~\ref{s:stsupermaps} we introduce stable supermaps using the notion of supercycle and push-forward of supercycle previously given. We also prove that the category fibred in groupoids of stable supermaps is a superstack in the sense of~\cite{CodViv17}.

In Section \ref{s:Norimotives} we start with a general formalism of Nori geometric categories and Nori diagrams as it is sketched in \cite{MaMar20}, and then present its specialisation in the context of supergeometry.
 Moreover  in Section \ref{ss:effNori}\footnote{This section  was  added during the revision process, after Professor Y.I. Manin passed away on January 7, 2023.}  we present three geometrical representations of SUSY Nori diagrams on abelian categories, thus providing concrete examples of effective mixed SUSY Nori motives as   in Definition \ref{def:effNori}.

We use notation and definitions from \cite{BrHRPo20} and references therein.
All schemes and superschemes   are \emph{locally noetherian} and, \blue{starting from Subsection \ref{ss:CM},}  all superscheme morphisms are \emph{locally of finite type}, and therefore also {\em locally of finite presentation.}
Starting from Section \ref{s:derham}, and throughout the rest of the paper  we assume that superschemes are locally of finite type over a field, usually the complex numbers.
 
\section{Basic algebraic supergeometry}\label{s:supergeom}
In this Section we collect  some basic definitions in supergeometry.  We adopt the following convention: for every $\Z_2$-graded module $M$, we write $M_+$ for its even part and $M_-$ for its odd part, so that $M=M_+\oplus M_-$. 

\subsection{Superrings and superschemes}
A \emph{superring} $\As$ is a $\Z_2$-graded supercommutative ring such that  
  the ideal  $J_\As$ generated by   odd elements  is finitely generated. Under this condition the graded ring $Gr_{J_\As}(\As)=\bigoplus_{i\geq 0} J_\As^i/J_\As^{i+1}$ is a finitely generated  module over  the \emph{bosonic reduction}  $A=\As/J_\As$ of $\As$. 
We shall say that $\As$ is \emph{split} if there exists a finitely generated projective $A$-module $M$ such that $\As\simeq {\bigwedge}_AM$. 

The notion of \emph{noetherian superring} generalizes the usual one, i.e., every ascending chain of $\mathbb Z_2$-graded ideals stabilizes \cite{We09}. 

\begin{defin}
A   locally ringed superspace is a pair $\Xcal=(X,\Oc_{\Xcal})$, where $X$ is a topological space, and $\Oc_{\Xcal}$ is a sheaf of $\Z_2$-graded commutative rings  such that for every point $x\in X$ the  stalk $\Oc_{\Xcal,x}$ is a local superring. 
A morphism of locally ringed superspaces is a pair $(f,f^\sharp)$, where $f\colon X \to Y$ is  a continuous map, and  $f^\sharp\colon \Oc_{\Sc}\to f_\ast\Oc_{\Xcal}$ is a homogeneous morphism of graded commutative sheaves, such that for every point $x\in X$, the induced morphism of 
local superrings
$\Oc_{\Sc,f(x)}\to \Oc_{\Xcal,x}$ is local.
\end{defin} 
 
If  $\Jc_\Xcal=(\Oc_{\Xcal,-})^2\oplus\Oc_{\Xcal,-}$
is  the homogeneous ideal   generated by the odd elements, $\Oc_X:=\Oc_{\Xcal}/\Jc_\Xcal$ is a purely even sheaf of rings, and the locally ringed space $X=(X,\Oc_X)$ is the \emph{ordinary locally ringed space underlying} $\Xcal$,  also called the \emph{bosonic reduction of} $\Xcal$ and  {sometimes} denoted   $\Xcal_{\bos}$.  There is a closed immersion of locally ringed superspaces
 {$i\colon X \hookrightarrow \Xcal$.}
A  {locally ringed superspace} $\Xcal$ is said to be \emph{projected} if there exists a morphism of locally ringed superspaces $\rho\colon \Xcal \to X$ such that $\rho\circ i=\Id$. As in the case of superrings, we said that $\Xcal$ is \emph{split} when $\Oc_\Xcal\simeq \bigwedge_{\Oc_X}\Ec$ for a locally free sheaf $\Ec$ on $X$ which generates the ideal $\Jc$, so that $\Ec\simeq\Jc_\Xcal/\Jc_\Xcal^2$. Analogously, $\Xcal$ is called \emph{locally split} if it can be covered by split open locally ringed superspaces.

The group $\Gamma=\{\pm1\}$ acts on $\Oc_\Xcal$ by $f^{(-1)}=f_+-f_-$ and defines another locally ringed superspace $\Xcal/\Gamma=(X, \Oc_{\Xcal,+})$, called the \emph{bosonic quotient} of~$\Xcal$.

The \emph{superspectrum} of a superring $\As$ is 
{the} {locally ringed superspace}  {$\SSpec\As=(X,\Oc)$}, where $X$ is the spectrum of the bosonic reduction $A$ of $\As$,
and $\Oc$ is a sheaf of $\Z_2$-graded commutative rings defined as follows:    any non-nilpotent element $f\in\As$ defines in the usual way a \emph{basic open subset}
$D(f)\subset X$, and one defines $\Oc(D(f))=\As_f$, the localization of $\As$ at the multiplicative subsystem defined by $f$.
The locally ringed superspaces of this form are called \emph{affine superschemes}.

\begin{defin}\label{def:superscheme} \label{def:dimension}
A superscheme is a locally ringed superspace $\Xcal=(X,\Oc_{\Xcal})$ which is locally isomorphic to the superspectrum of a superring. 

A superscheme $\Xcal=(X,\Oc_{\Xcal})$ is \emph{noetherian} if $X$ has a finite open cover $\{U_i\}$ such that every restriction $\Xcal_{\vert U_i}$ is the superspectrum of a noetherian superring.  
\end{defin}

When $\Xcal$ is a noetherian superscheme, the ideal sheaf $\Jc_\Xcal$ is coherent, and   $\Jc_\Xcal/\Jc_\Xcal^2$ is a coherent $\Oc_X$-module.

 {The even dimension of $\Xcal$ is the dimension of the scheme $X$. When $\Xcal$ is locally split, we define its odd dimension as the rank of the locally free $\Oc_X$-module $\Jc_\Xcal/\Jc_\Xcal^2$ and the dimension of $\Xcal$ as  the pair $\dim\Xcal=\operatorname{even-dim} \Xcal \,\vert\, \operatorname{odd-dim}\Xcal$.
There are various possibilities for a definition of the odd dimension of a superscheme that generalize the one given for split superschemes (see \cite[7.1.1]{We09} for the affine case). In this paper we shall use the following.
\begin{defin}\label{def:dim} The odd dimension of a superscheme $\Xcal$ is  the smallest integer number $n$ such that  $\Jc_\Xcal^{n+1}=0$, that is, the odd dimension is the order according to \cite[2.1]{BrHRPo20}.  We say that $\Xcal$ 
has \emph{pure} odd dimension $n$ if 
\begin{enumerate}
\item $\Xcal$ has odd dimension $n$, and
\item the $\Oc_X$-module $\Jc_\Xcal/\Jc_\Xcal^2$ is generically of rank $n$, that is, it is of rank $n$ on a dense open subset of every irreducible component of $X$.
\end{enumerate}
\end{defin}
}
We say that a morphism $f\colon \Xcal\to \Sc$ of superschemes has \emph{relative dimension $m\vert n$} if the fibres $\Xcal_s$ are superschemes of dimension $m\vert n$.

\subsubsection{Irreducible superschemes}

A superring $\As$ is \emph{integral} or it is a \emph{superdomain} \cite{We09} (resp. reduced, resp. a \emph{superfield})   if its bosonic reduction  is an integral commutative ring (resp. reduced, resp. a field). 
If $\As$ is a superdomain, its superring of total fractions is the localization $\K(\As)=S^{-1}\As$ by the multiplicative system of the even elements that are not zero divisors. Note that the non zero divisors always belong to $\As-J$; we then easily see that $\K(\As)$ is a superfield, because its bosonic reduction is the field of fractions $K(A)$ of $A$.

We say that a superscheme $\Xcal=(X,\Oc_\Xcal)$ is \emph{integral} (resp. \emph{reduced}) if the superring $\Oc_\Xcal(U)$ is a superdomain (resp. a reduced superring) for every affine open sub-superscheme $\U$. This is equivalent to its bosonic reduction $X$ to be an integral (resp. reduced) scheme.
If $\Xcal$ is an integral superscheme, we can define its superring $\K(\Xcal)$ of rational functions as the local superring $\Oc_{\Xcal,x_0}$, where $x_0$ is the generic point of $X$. We have   $\K(\Xcal)=\K(\Oc_\Xcal(U))$ for every affine open sub-superscheme $\U$.

In the sequel we shall use the words \emph{supervariety} and \emph{sub-supervariety} as  synonyms for \emph{integral superscheme} and \emph{integral closed sub-superscheme}, respectively.

\subsection{Separated, proper and  flat morphisms of superschemes}\label{ss:propermor}\label{s:smoothmorphisms}
The definition of some types of scheme morphisms can be straightforwardly generalized to the superscheme setting.
\begin{defin}\label{def:proper}
A morphism $f\colon \Xcal\to\Sc$ of superschemes is:
\begin{enumerate}
\item affine, if for every affine open sub-superscheme $\Ucal\subset \Sc$ the inverse image $f^{-1}(\Ucal)$ is affine;
\item finite, if it is affine, and for $\Ucal = \SSpec \As\subset \Sc$, then $f^{-1}(\Ucal)=\SSpec \Bs$, where $\Bs$ is a finitely generated graded $\As$-module;
\item locally of finite type, if $\Sc$ has an affine open cover $\{\Ucal_i=\SSpec \As_i\}$ such that every inverse image
 $f^{-1}(\Ucal_i)$ has an open affine cover $\mathfrak V_i =\{\mathcal V_{ij} = \SSpec \Bs_{ij}\}$, where
 each $\Bs_{ij}$ is a finitely generated graded $\As_i$-algebra.
 \item $f$ is of finite type if in addition each $\mathfrak V_i$ can be taken to be finite. In other words, if it is locally of finite type and quasi-compact.
\item separated, if the diagonal morphism $\delta_f\colon \Xcal\to \Xcal\times_\Sc \Xcal$ is a closed immersion (actually it is enough to ask that it is a closed morphism);
\item proper, if it is separated, of finite type and universally closed;
\item flat, if for every point $x\in X$, $\Oc_{\Xcal,x}$ is flat over $\Oc_{\Sc,f(x)}$;
\item faithfully flat, if it is flat and surjective.
\end{enumerate}
\end{defin}

As usual,   flatness and faithful flatness are stable under  base change and composition.
Separatedness and properness   only depend  on the   morphisms between the underlying ordinary schemes
 {(\cite[Prop.~A.13]{BrHRPo20})}:

\begin{prop}\label{prop:proper}
A morphism $f\colon \Xcal\to\Sc$ of superschemes which is  {locally of finite type}  is proper (resp.~separated) if and only if the induced scheme morphism $f_{\bos}\colon X \to S$ is proper (resp.~separated).
\end{prop}

As a consequence, the valuative criteria for separatedness  and properness \cite[Thm.~II.4.3 and II.4.7]{Hart77}   still apply in the graded setting.
 Analogously, we can then extend to superschemes many of the properties of proper and separated morphisms of schemes.

\begin{prop}\label{prop:sorite}
{\ }
\begin{enumerate}
\item The properties of being separated and proper are stable under  base change.
\item Every morphism of affine superschemes is separated.
\item Open immersions are separated and closed immersions are proper.
\item The composition of two separated (resp.~ proper) morphisms of superschemes is separated (resp.~ proper).
\item If $g\colon \Ycal \to \Xcal$ and $f\colon \Xcal \to \Sc$ are morphisms of superschemes and $f\circ g$ is proper and {$f$ is separated, then $g$ is proper. If  $f\circ g$ is separated, then $g$ is separated.} 
\end{enumerate}\vspace{-2em}
\qed
\end{prop}

A suitable notion of smoothness of a morphism of superschemes turns out to be the following. For a motivation see \cite{BrHRPo20}, Section A.6.
\begin{defin}\label{def:smooth}  A morphism $f\colon \Xcal\to\Sc$ of superschemes of relative dimension $m\vert n$ is smooth if:
\begin{enumerate}
\item $f$ is locally of finite presentation (if the superschemes are locally noetherian, it is enough to ask that $f$ is locally of finite type);
\item $f$ is flat;
\item the sheaf of relative differentials $\Omega_{\Xcal/\Sc}$ is locally free of rank $m\vert n$.
\end{enumerate}
\end{defin}

When $\Sc=\Spec k$  one has the notion of smooth  
superscheme over a field.
One has the following criterion for smoothness over a field. For a proof see \cite[Prop.~A.17]{BrHRPo20}.

\begin{prop}\label{prop:smoothsplit}
 A superscheme $\Xcal$  of dimension $m\vert n$,
locally of finite type over a field $k$,   is smooth if and only if
\begin{enumerate}
\item $X$ is a smooth scheme over $k$ of dimension $m$;
\item  the $\Oc_X$-module $\Ec=\Jc/\Jc^2$ is locally free of rank $n$ and the natural map ${\bigwedge}_{\Oc_X} \Ec\to Gr_J(\Oc_\Xcal)$ is an isomorphism.
\end{enumerate}
If $\Xcal$ is smooth of dimension $m\vert n$ then it is locally split,
and for every closed point $x\in X$ there exist \emph{graded local coordinates}, that is, $m$ even functions $(z_1,\dots,z_m)$ which generate the maximal ideal $\mathfrak m_x$ of $\Oc_{X,x}$ and $n$  odd functions $(\theta_1,\dots,\theta_n)$ generating $\Ec_x$, such that $(dz_1,\dots,dz_m,d\theta_1,\dots,d\theta_n)$ is a   basis for $\Omega_{\Xcal,x}$.
\end{prop}

\begin{rem}(Formally smooth morphisms)
As in the commutative case, the smoothness of morphisms may be defined in different ways.
One may say that 
a morphism  of superschemes $f\colon \Xcal\to \Sc$
is formally smooth if for every affine $\Sc$-superscheme $\SSpec(\Bs)$ and every nilpotent ideal $\Nc\subset \Bs$,
any $\Sc$-morphism $\SSpec(\Bs/\Nc)\to \Xcal$ extends to an $\Sc$-morphism $\SSpec(\Bs)\to \Xcal$. 
However, it turns out that in the locally noetherian case, a superscheme morphism which is locally
of finite type is formally smooth if and only if it is smooth (see \cite{BrHRPo20}, Section A.6.1).
\end{rem}

\subsection{(Quasi-)coherent  $\Oc_\Xcal$-modules and cohomology}
If  $\M$ is an $\Oc_\Xcal$-module on a superscheme $\Xcal=(X,\Oc_\Xcal)$ one can   consider sheaf
cohomology groups $H^\bullet(X,\M)$, with $\M$ 
regarded as a sheaf of abelian groups on $X$, so that 
  the usual cohomology vanishing   also holds: if $\Xcal$ is noetherian, then  $H^i(X,\M)=0$  for every  $i>\dim X$.
  
  The notion of (quasi-)coherent  $\Oc_\Xcal$-module can be introduced by mimicking the ordinary theory. 
 Moreover,     a notion of \emph{projective superscheme} can be introduced as well; here some new features are involved,
 as for instance super-Grassmannians turn out not to be projective.
 
 For coherent sheaves on a projective superscheme one has a finiteness result.

\begin{prop} If $\Xcal$ is a projective superscheme over a field $k$ and $\M$ is a coherent sheaf on it, then all groups $H^i(X,\M) $ are finite-dimensional over $k$.
\label{prop:finiteness}
\end{prop}

 More generally, there is a finiteness result for the cohomology of proper morphisms.
\begin{prop} \label{prop:finiteness2} Let $f\colon \Xcal \to \Sc$ be a a proper morphism of superschemes. For every coherent sheaf $\M$ on $\Xcal$ and for every $i\geq 0$, the higher direct images $R^i f_\ast\M$ are coherent sheaves on $\Sc$.
\end{prop}

For details and proofs about these notions and results see \cite{BrHRPo20}.

\subsubsection{Zariski's Main theorem}
Let $f\colon \Xcal \to \Sc$ be a proper morphism of superschemes. By Proposition \ref{prop:finiteness2} $f_\ast\Oc_\Xcal$ is a coherent sheaf of $\Oc_\Sc$-superalgebras. Then, the natural morphism $g\colon \Ycal=\SSpec f_\ast\Oc_\Xcal\to \Sc$ is finite and $f$ factors as
$$
\xymatrix{\Xcal \ar[rr]^{\bar f}\ar[rd]^f & & \Ycal \ar[ld]^g \\
& \Sc&}
$$
Moreover, the fibres of $\bar f$ are connected. This is the super version of the \emph{Stein factorization} \cite[4.3.1]{EGAIII-I}.

One easily see that the proofs of \cite[Prop.\, 4.4.1 and 4.4.2]{EGAIII-I} are still valid for superschemes. One then has:
\begin{prop}[Zariski's Main Theorem]\label{prop:zariski} Let $f\colon \Xcal \to \Sc$ be a  morphism of superschemes. $f$ is finite if and only if it is proper and  its bosonic fibres are finite.
\end{prop}

\subsection{\'Etale topology and faithfully flat descent for superschemes}

If $f\colon \Xcal\to\Sc$ is a flat morphism of superschemes, then $f_\bos\colon X\to S$ is also flat. If $f$ is locally of finite presentation, this implies that it is \emph{universally open}, and that the image $\Ima f$ is open in $S$.

\begin{defin}\label{def:etale}
A morphism $f\colon \Xcal\to\Sc$ is \'etale if it is smooth of relative dimension $(0,0)$ over the open sub-superscheme $\Ima f$ of $\Sc$.
An \'etale covering is a surjective \'etale morphism.
\end{defin}

If $f\colon \Xcal\to\Sc$ is smooth or \'etale, the induced morphism $f\colon X\to S$ between the underlying ordinary schemes is smooth or \'etale as well.

\begin{defin}[\cite{DoHeSa97}, Def.\,7]\label{def:etaletop}
The \emph{\'etale topology} is the Grothendieck topology on the category $\Sf$ of superschemes whose coverings are the surjective \'etale morphisms.
\end{defin}

This allows to generalize to supergeometry some standard constructions and definitions  about the \'etale topology of schemes. In particular, we have the following definition (see \cite{DoHeSa97}).
\begin{defin}\label{def:etalerel} 
An (Artin) algebraic superspace is a sheaf $\Xf$ for the \'etale topology of superschemes that can be expressed as the categorical quotient of an \'etale equivalence relation of superschemes
$$
\Tc \rra \U \to \Xf\,.
$$
(here we are confusing superschemes and their functors of  points).
\end{defin}

As in ordinary geometry, the notion of (Artin) algebraic superspace is quite useful when it comes to study moduli problems for superschemes. One nice example is the construction of the moduli or RR-SUSY curves. Under certain conditions, the functor of relative RR-SUSY curves of genus $g$ along a (nonramified) RR puncture of degree $\nf_R$   with $\nf_{NS}$ NS punctures is representable by a separated algebraic superspace $ \Xf_{\nf_{NS},\nf_{R}}$ \cite[Thm.\,5.1, 5.2]{BrHR21}. Later on we shall come back to the construction of moduli spaces of supercurves and maps as superstacks (Theorem \ref{thm:moduli} and Section \ref{s:stsupermaps}).
Actually,  as in ordinary geometry, some moduli problems cannot be represented by (Artin) algebraic superspaces, and one has to resort to more general structures. Usually one treats these ``supermoduli spaces'' as stacks on the category of superschemes; a description of superstacks is given in \cite{CodViv17}.

For various operations between superschemes, algebraic superspaces and stacks some notions of local character and descent properties    are necessary. We review here some of them.

\begin{defin} A property P of  morphisms of superschemes is local on the target  for the \'etale topology if a morphism $f\colon \Xcal \to \Sc$ has the property P if and only if for every \'etale covering $\phi\colon \Tc\to \Sc$, the fibre product $\phi^\ast f\colon \Xcal\times_\Sc \Tc \to \Tc$ has the property P.

A property P of morphisms of superschemes is local on the source  for the \'etale topology if a morphism $f\colon \Xcal \to \Sc$ has the property P  if and only if for every \'etale covering $\phi\colon \Tc\to \Xcal$, the composition $f\circ\phi\colon \Tc \to \Sc$ has the property P. 
\end{defin}

\begin{prop}\label{p:etalelocal}
The properties of being flat, smooth and \'etale, are local on the target and on the source for the \'etale topology of superschemes. The properties of being separated and proper are local on the target for the same topology.
\end{prop}
\begin{proof}
The properties of being flat, separated, proper,  smooth and \'etale are stable  under  base change.  The remaining   properties follow from the fact that \'etale covers are faithfully flat. 
\end{proof}

Grothendieck's   faithfully flat descent   \cite[Exp.~ VIII]{SGA1} can be  quite easily extended to superschemes.  We state here the main results.

\begin{prop}[Descent for homomorphisms]\label{prop:morphdescent} Let $p\colon \Tc\to \Sc$ be a faithfully flat quasi-compact morphism of superschemes, $\Rcal=\Tc\times_\Sc\Tc$. Denote by $p_1,p_2\colon \Rcal\rra\Tc$ the projections. Let $\M$, $\Nc$ be (graded) quasi-coherent sheaves on $\Sc$, $\M'=p^\ast\M$, $\Nc'=p^\ast\Nc$ and $\M''=p_1^\ast\M'=p_2^\ast\M'$, $\Nc''=p_1^\ast\Nc'=p_2^\ast\Nc'$. The sequence 
$$
\Hom_\Sc(\M,\Nc) \to \Hom_\Tc(\M',\Nc') \rra \Hom_\Rcal(\M'',\Nc'')
$$
of (homogeneous) homomorphisms induced by the projections is exact, namely,  there is a one-to-one correspondence between $\Hom_\Sc(\M,\Nc)$ and the coincidence locus  of the pair of arrows.
\qed \end{prop}

\begin{prop}[Descent for modules]\label{prop:moddescent} Let $p\colon \Tc\to \Sc$ be a faithfully flat quasi-compact morphism of superschemes, and let  $\Rcal=\Tc\times_\Sc\Tc$ and $(p_1,p_2)\colon \Rcal\rra\Tc$ be  the projections. Let $\M$ be a graded quasi-coherent sheaf on $\Tc$ with  descent data, that is, an isomorphism  $\phi\colon p_1^\ast\M\iso p_2^\ast \M$ such that,
if $\U = \Tc\times_\Sc\Tc\times_\Sc\Tc$, and $p_{12}$, $p_{23}$, $p_{13}$ are the three projections onto $\Rcal$, the condition
$p_{23}^\ast \phi \circ p_{12}^\ast \phi = p_{13}^\ast\phi$
holds. Then, there is a graded quasi-coherent sheaf $\Nc$ on $\Sc$ such that $\M\simeq p^\ast\Nc$.
\qed \end{prop}
\marginnote{Sistemati errori di battitura}

 For every superscheme $\Xcal$  we denote by $H(\Xcal)$ the set of all   closed sub-superschemes of $\Xcal$.
\begin{prop}[Descent for closed sub-superschemes]\label{prop:subdescent}
 Let $p\colon \Tc\to \Sc$ be a faithfully flat quasi-compact morphism of superschemes, The   sequence of sets
 $$
 H(\Sc)\to H(\Tc) \rra H(\Rcal)
 $$
 is exact.
\qed\end{prop}

\subsection{Cohen-Macaulay and Gorenstein morphisms}\label{ss:CM}

\blue{We recall that from this point on, all morphisms are locally of finite type, and then locally of finite presentation as we only consider locally noetherian superschemes.}

For any morphism $f\colon \Xcal\to\Sc$ of superschemes we consider the diagram 
\begin{equation}\label{eq:diagram}
\xymatrix{
X  \ar@{^{(}->}[r]^j\ar[rd]_{f_{bos}}& \Xcal_S \ar@{^{(}->}[r] ^i\ar[d]^{f_S} & \Xcal\ar[d]^f\\
& S  \ar@{^{(}->}[r] ^i&\Sc
}
\end{equation}
where $X$ and $S$ denote the  underlying  bosonic schemes.

By Grothendieck duality for superschemes  \cite[Sec.\,3.5]{BrHRPo20}, if $\Xcal$ and $\Sc$ are quasi-compact and separated an $f$ is quasi-compact and separated as well, the direct image functor $\bR f_\ast\colon D(\Xcal)\to D(\Sc)$  between the derived categories of (super) quasi-coherent sheaves has a right adjoint  $f^!\colon D(\Sc)\to D(\Xcal)$ \cite[Prop. 3.22]{BrHRPo20}\footnote{In recent literature about Grothendieck duality for schemes, the right adjoint to $\bR f_\ast$ is denoted by $f^\times$ and the notation $f^!$ is used for a pseudofunctor that coincides with $f^\times$ for proper morphisms and with $f^\ast$ when $f$ is an open immersion}. The object $\Dcal^\bullet_f:=f^!\Oc_\Sc$ in $D(\Xcal)$ is called the \emph{dualizing complex} of $f$.

A quite useful elementary property is the compatibility with flat base change: whenever one has a cartesian diagram
$$
\xymatrix{\Xcal_\Tc \ar[r]^{\phi_\Xcal}\ar[d]^{f_\Tc} & \Xcal \ar[d]^f \\
\Tc \ar[r]^\phi & \Sc
}
$$
where $\phi$ is \emph{flat} and $f$ is  proper, 
there is a functorial isomorphism
$$
f_\Tc^!(\phi^\ast \Nc^\bullet)\simeq \phi_\Xcal^\ast f^!(\Nc^\bullet) 
$$
for any $\Nc^\bullet$ in $D(\Sc)$ of finite homological dimension \cite[Prop. 3.27]{BrHRPo20}. In particular
\begin{equation}\label{eq:flatbasechange}
f_\Tc^!(\Oc_\Tc)\simeq \phi_\Xcal^\ast f^!(\Oc_\Sc) =  f^!(\Oc_\Sc)_{\Xcal_\Tc}\,.
\end{equation}

In the purely bosonic case, a flat and proper scheme morphism $f_{bos}\colon X\to S$ of relative dimension $m$ has the property that the geometric fibres $X_s$ are Cohen-Macaulay (CM) schemes if and only if 
the cohomology sheaves of the dualizing complex $f_{bos}^!\Oc_S$ vanish, except the $(-m)$-th one, and this is flat over $S$. Recall that a scheme $X$ of dimension $m$ is CM if and only if $\Ext^i_{\Oc_{X,x}}(\kappa(x), \Oc_{X,x})=0$ for all $i\neq m$ and every (closed) point  $x\in X$. This  is the motivation for the following definition.

\begin{defin}\label{def:CM}  A morphism $f\colon \Xcal\to\Sc$ of superschemes is Cohen-Macaulay (CM) if it is proper\footnote{We assume that $f$ is proper because our definition of CM morphism is given in terms of the dualizing sheaf.}and flat and  $$
f^!\Oc_\Sc\simeq \omega_f[m]
$$
for a  \blue{coherent}  sheaf $\omega_f$ which is flat over $\Sc$, where $m$ is the relative even dimension. The sheaf $\omega_f$
is called the \emph{relative dualizing sheaf}. A morphism $\pi\colon \Xcal\to\Sc$ of superschemes is Gorenstein if it is CM and the relative dualizing sheaf is a line bundle.
\end{defin}

We  use also the notation $\omega_{\Xcal/\Sc}$ for $\omega_f$. When $\Sc$ reduces to a single point $s$ we will write $\omega_{\Xcal_s}$.

 A smooth morphism is Gorenstein and its dualizing sheaf is the relative Berezinian \cite{OgPe84, Penk83}, \cite[Prop. 3.35]{BrHRPo20}.

A first consequence of Equation \eqref{eq:flatbasechange} is the following:
\begin{lemma}
For a given proper, flat and  surjective  (i.e., faithfully flat) morphism $\Tc\to\Sc$, the morphism $f\colon \Xcal\to\Sc$ is CM (resp.~Gorenstein) if and only if $f_\Tc\colon \Xcal_\Tc\to\Tc$ is CM (resp.~Gorenstein). 
\qed
\end{lemma}

This is the case when $\Tc\to\Sc$  is a Zariski or   \'etale covering. One then has:
\begin{prop}
For a proper morphism, the condition of being CM (resp. Gorenstein) is local on the base, for both the Zariski and the \'etale topologies.
\qed
\end{prop}

The relationship between the dualizing complexes of $f_S\colon \Xcal \to S$ and $f_{bos}\colon X \to S$ is given by
$$
f_{bos}^!(\Oc_S)\simeq \bR\Homsh_{\Oc_{\Xcal_S}}(\Oc_X,f_S^!(\Oc_S))\,.
$$
The involved Ext sheaves are not easily computed. 

Some results about  the functor $f^!$ and compatibility with base change can be strengthened for proper superscheme morphisms.
Let us  start by noticing that arbitrary base change compatibility for the push-forward is still true when the morphism is flat. Namely, given a cartesian diagram of superscheme morphisms
$$
\xymatrix{ \Xcal_\Tc\ar[r]^{\phi_\Xcal}\ar[d]^{f_\Tc}& \Xcal \ar[d]^f
\\
\Tc\ar[r]^\phi& \Sc
}
$$
where either $\phi$ or $f$ is \emph{flat}, the base change morphism is an isomorphism in the derived category $D(\Tc)$:
\begin{equation}\label{eq:nonflatbc}
\bL\phi^\ast \bR f_\ast \simeq \bR f_{\Tc\ast}\bL \phi_\Xcal^\ast\,.
\end{equation}

This can be proven as in \cite[A.85]{BBH91}.

\begin{prop}\label{prop:f!} Let $f\colon\Xcal\to\Sc$ be a proper morphism of superschemes.
For every complex $\Nc^\bullet$ in $D(\Sc)$, there is an isomorphism
$$
f^!\Nc^\bullet\iso \bL f^\ast\Nc^\bullet\otimes^{\bL} f^!\Oc_\Sc
$$ 
in the derived category $D(\Xcal)$ (compare with \cite[Prop.\,3.27]{BrHRPo20}). Moreover, if $\Nc^\bullet$ have coherent cohomology sheaves, $f^!\Nc^\bullet$ has coherent cohomology sheaves as well.
\end{prop}
\begin{proof}
By \cite[Prop.\,3.25 and 3.26]{BrHRPo20} we can assume that $\Sc=\SSpec\As$ is affine and that $f$ is superprojective (here we use that $f$ is locally of finite type). If $\Lcl$ is a relatively very ample line bundle on $\Xcal$, both $\Lcl$ and $f^!\Oc_\Sc$ are of finite homological dimension. For any pair $(r,s)$ of integers, one has
\begin{align*}
&\bR f_\ast\bR\Homsh_{\Oc_\Xcal}(\Lcl^r[s], \bL f^\ast \Nc^\bullet\otimes^\bL f^!\Oc_\Sc)\\
	     &\overset{(1)}\simeq \bR f_\ast(\bR\Homsh_{\Oc_\Xcal}(\Lcl^r[s],f^!\Oc_\Sc)\otimes^\bL \bL f^\ast\Nc^\bullet)\\[2pt]
&\overset{(2)}\simeq  \bR f_\ast\bR \Homsh_{\Oc_\Xcal}(\Lcl^r[s],  f^!\Oc_\Sc)\otimes^\bL \Nc^\bullet
 \simeq  \bR \Homsh_{\Oc_\Sc}(\bR f_\ast \Lcl^r[s], \Oc_\Sc)\otimes^\bL \Nc^\bullet 
\\[2pt]
&\overset{(3)}\simeq  \bR \Homsh_{\Oc_\Sc}(\bR f_\ast \Lcl^r[s], \Nc^\bullet)
\simeq  \bR f_\ast \bR\Homsh_{\Oc_\Xcal}(\Lcl^r[s], f^!\Nc^\bullet)
\end{align*}
where $(1)$ and (3) follow from \cite[Eq.\,3.19]{BrHRPo20} because  $\Lcl$ is of finite homological dimension over $\Oc_\Xcal$ and $\bR f_\ast \Lcl^r[s]$ is of finite homological dimension over $\As$ (see \cite[Prop.\,2.33]{BrHRPo20}) and $(2)$ is the projection formula, that can be proved as in \cite[Prop.\,5.3]{Nee96}.
Then $\Hom_{D(\Xcal)}(\Lcl^r[s],\Cc^\bullet)=0$ for every $r,s$, 
where $\Cc^\bullet$ is the cone of the natural morphism $f^! \Nc^\bullet\to \bL f^\ast\Nc^\bullet\otimes^\bL f^!\Oc_\Sc$. Proceeding as in \cite[Example 1.10]{Nee96}, one sees that the family of  objects $\Lcl^r[s]$ is a compact generating set of $D(\Xcal_\Tc)$, so that one gets $\Cc^\bullet=0$ in $D(\Xcal)$ and then $f^! \Nc^\bullet\to \bL f^\ast\Nc^\bullet\otimes^\bL f^!\Oc_\Sc$ is an isomorphism.

Assume that $\Nc^\bullet$ has coherent cohomology. Since one has $f^!\Nc^\bullet\iso \bL f^\ast\Nc^\bullet\otimes^{\bL} f^!\Oc_\Sc$ one has only to prove that $f^!\Oc_\Sc$ has coherent cohomology. Again by \cite[Prop.\,3.25 and 3.26]{BrHRPo20} we can assume that $\Sc=\SSpec\As$ is affine and $f$ is superprojective. Then $f$ factors as a closed immersion $\Xcal \hookrightarrow \Ps_\As^{p|q}$ and the natural projection $\pi\colon \Ps_\As^{p|q}\to \Sc$. If we denote simply by $\Oc$ the structure sheaf of $\Ps_\As^{p|q}$, one has  $f^! \Oc_\Sc \simeq \bR \Homsh_\Oc (\Oc_\Xcal, \Oc(q-p-1)[p])$ by \cite[Corol.\,3.36 and Prop.\,3.9]{BrHRPo20}, thus finishing the proof.
\end{proof}

Let $f\colon\Xcal\to\Sc$ be a morphism of superschemes. For every point $s\in\Sc$ we can consider the base change to the fibre
$$
\xymatrix{\Xcal_s \ar[d]^{f_s}\ar@{^{(}->}[r]^{j_s} & \Xcal \ar[d]^f\\
s=\Spec\kappa(s)\ar@{^{(}->}[r]^(.7){j_s} & \Sc
}
$$
\begin{prop}\label{prop:arbitrarybc} If $f$ is flat and proper,  the natural base change morphism
$$
\bL j_s^\ast \Dcal_f^\bullet =\bL j_s^\ast f^! \Oc_\Sc \to f_s^! \kappa(s)=\Dcal_{f_s}^\bullet
$$
is an isomorphism in $D(\Xcal_s)$.
\end{prop}
\begin{proof} By Proposition \ref{prop:f!}, one has $\bL j_s^\ast f^!\Oc_\Sc\simeq f^!\Oc_\Sc\otimes_{\Oc_X}^\bL f^\ast \kappa(s)\simeq f^!(\kappa(s))$. Moreover, since $f$ is flat, we can apply Equation \eqref{eq:nonflatbc}, and for every complex $\Nc^\bullet$ in $D(\Xcal)$, one has
\begin{align*}
\Hom_{D(\Xcal)}(\Nc^\bullet, j_{s\ast}f_s^! \kappa(s)) &\simeq \Hom_{D(\Xcal_s)}(\bL j_s^\ast \Nc^\bullet, f_s^! \kappa(s)) \\
						       &\simeq \Hom_{D(\kappa(s))}(\bR f_{s\ast}\bL j_s^\ast \Nc^\bullet, \kappa(s))
\\
& \simeq \Hom_{D(\kappa(s))}(\bL j_s^\ast \bR f_{\ast}\Nc^\bullet, \kappa(s))\\
&\simeq \Hom_{D(\Sc)}(\bR f_{\ast}\Nc^\bullet, \kappa(s))\,,
\end{align*}
so that $j_{s\ast}f_s^! \kappa(s)\simeq f^!\kappa(s)$ and the result follows.
\end{proof}
 
\begin{prop} Assume that  $f\colon\Xcal\to\Sc$ is 
 flat and proper. Then it is CM (resp. Gorenstein) if and only if for every point $s\in \Sc$ the fibre $\Xcal_s$ is a CM (resp. Gorenstein) superscheme. Moreover, in such a case, the natural morphism $\omega_{f\vert \Xcal_s}\to \omega_{\Xcal_s}$ is an isomorphism.
\end{prop}
\begin{proof} If $f$ is CM, $f^!\Oc_\Sc\simeq \omega_f[m]$ where $m$ is the even dimension of $f$ and $\omega_f$ is flat over $\Sc$. By Proposition \ref{prop:arbitrarybc}, the dualizing complex of the fibre is 
$$
\Dcal^\bullet_{\Xcal_s}\simeq \bL j_s^\ast \omega_f[m]\simeq j_s^\ast \omega_f[m]
$$
so that $\Xcal_s$ is CM and its dualizing sheaf is $\omega_{f |\Xcal_s}$. 

For the converse we can assume that $\Sc$ is the superspectrum of a local superring $\As$ and write $\kappa$ for the residual field. We have $j_{s\ast}f_s^!\kappa\simeq f^!\kappa$. For every point $x\in X_s$ 
and for every index $i$, we consider the half exact functor $T^i (M)=\Hc^{-i}(f^! M)_x$ from the category of finitely generated $\As$-modules to itself (the fact that the $\As$-modules $T^i (M)$ are finitely generated follows from the second part of Proposition \ref{prop:f!}. Then $T^i(\kappa)=\Hc^{-i}(f^! \kappa)_x\simeq \Hc^{-i}(f_s^! \kappa)_x=0$ for $i\neq m$ so that $T^i=0$ for Nakayama for half exact functors \cite[Lemma A.9]{BrHRPo20}; in particular $\Hc^{-i}_x(f^!\As)=T^i(\As)=0$ for $i\neq m$. 

Moreover, $T^{m-1}=0$ implies that $T^m(\As)\to T^m(\kappa)$ is surjective, and then $T^m(\As)\otimes \  \simeq T^m$ by \cite[Prop.\,A.10]{BrHRPo20}. Since, since $T^{m+1}=0$, the functor $T^m$ is left exact, so that $T^m(\As)$ is flat over $\As$. This proves the converse of our statement. 

The Gorenstein case is similar.
\end{proof}
One then has
\begin{prop}\label{prop:bcCM} Let $f\colon\Xcal\to\Sc$ be a CM (resp. Gorenstein) morphism of superschemes. For every base change $\Tc\to\Sc$, the induced morphism $f_\Tc\colon \Xcal_\Tc\to \Tc$ is CM (resp. Gorenstein) and $\omega_\Tc\simeq \omega_{\Sc\vert \Xcal_\Tc}$.
\qed
\end{prop}

\subsection{Analytic superspaces and differentiable supermanifolds}\label{ss:gaga}
In this Section we extend the Definition \ref{def:superscheme} of superscheme to cover the analytic and differentiable environments. We assume that the base field is  $\C$. Let $\Xcal=(X,\Oc_{\Xcal})$ be  a locally ringed superspace and $(X,\Oc_X)$ its bosonic reduction.
\begin{defin}\label{def:superanalytic} $\Xcal=(X,\Oc_{\Xcal})$ is said to be an analytic superspace if the bosonic reduction $X=(X,\Oc_X)$ is a complex analytic space.
\end{defin}

Then, $X$ has the usual euclidean topology and $\Oc_X$ is the sheaf of holomorphic functions.
Most of the definitions and basic properties of superschemes are valid  for analytic superspaces with the obvious modifications, in particular, smoothness (Definition \ref{def:smooth}) and its local characterization (Proposition \ref{prop:smoothsplit}).

Complex superschemes can be considered as analytic superspaces. More precisely, if $\Xcal$ is a complex superscheme locally of finite type, proceeding as in \cite[Exp.\,XII]{SGA1}, one can associate to it an analytic superspace $\Xcal^{an}$ characterized by a universal property: the functor 
$$
\Ycal \mapsto \Hom(\Ycal,\Xcal)
$$
on the category of analytic superspaces, where $\Hom$ stands for the homomorphisms of complex locally ringed superspaces, is representable by an analytic superspace $\Xcal^{an}$. One then has 
\begin{equation}\label{eq:an}
 \Hom(\Ycal,\Xcal)\simeq  \Hom(\Ycal,\Xcal^{an})\,,
\end{equation}
for every analytic superspace $\Ycal$. Moreover, for every closed point $x$ one has an isomorphisms $ \widehat\Oc_{\Xcal^{an},x}\simeq\widehat\Oc_{\Xcal,x}$ between the completions with respect to the topologies defined by the corresponding maximal ideals. By Equation \eqref{eq:an} the identity on $\Xcal^{an}$ induces a morphism $\phi_{\Xcal}\colon\Xcal^{an}\to\Xcal$ of locally ringed superspaces. Moreover, associating $\Xcal^{an}$ to $\Xcal$ is a functor, and given a morphism $f\colon \Xcal \to \Sc$ of superschemes, there is a commutative diagram
$$
\xymatrix{ \Xcal\ar[r]^f& \Sc\\
\Xcal^{an}\ar[r]_{f^{an}}\ar[u]^{\phi_{\Xcal}} & \Sc^{an}\ar[u]^{\phi_{\Sc}}\,.
}
$$
The bosonic reduction of $\Xcal^{an}$ is the analytic space $X^{an}=(X^{an},\Oc^{an}_{X^{an}})$ associated to the scheme $(X,\Oc_X)$ by \cite[Exp.\,XII]{SGA1}. As a set $X^{an}\simeq X^{an}(\C)=\Xcal^{an}(\C)$ is the set of the closed (or geometric) points of both $\Xcal^{an}$ and $X^{an}$; it is endowed with the usual Euclidean topology.
\begin{defin}\label{def:analytization} The analytic superspace $\Xcal^{an}$ is   the analytification of the superscheme $\Xcal$.
\end{defin}

$\Xcal$ is a smooth superscheme if and only if $\Xcal^{an}$ is a smooth analytic superspace; moreover, the analytic superspace associated to the affine space $\As_\C^{m|n}$ is $\C^{m|n}=(\C^m, \Oc^{an}[\theta_1,\dots,\theta_n])$, where $\Oc^{an}$ is the sheaf of holomorphic functions of $\C^m$.

If $\Xcal$ is a superscheme, for every sheaf $\M$ of $\Oc_{\Xcal}$-modules, we denote by $\M^{a}$ the pullback $\phi_{\Xcal}^\ast \M$ on the analytic superspace $\Xcal^{an}$. It is clear that if $\M$ is quasi-coherent (resp. coherent) on $\Xcal$, then $\M^{an}$ is also quasi-coherent (resp. coherent) on $\Xcal^{an}$. Moreover, the super version of Serre's GAGA theorem on cohomology (see \cite[Exp.\,XII]{SGA1}) is also true:

\begin{prop}[Super GAGA]\label{prop:GAGA} Let $f\colon \Xcal\to \Sc$ be a proper morphism of complex superschemes and $\M$ a coherent $\Oc_\Xcal$-module. For every integer $i\geq 0$, the natural morphism
$$
(R^i f_\ast\M)^{an} \to R^i f^{an}_\ast (\M^{an})
$$
of sheaves on $\Sc^{an}$ is an isomorphism.
\end{prop}
\begin{proof} Let $\Ic$ be the ideal of $X$ as a closed super subscheme of $\Xcal$. There is a finite filtration 
$$
0\subset {\Ic^{n}}\M\subset \dots\subset \Ic\M\subset \M\,,
$$
whose successive quotients $\M^{(p)}= \Ic^p\M/\Ic^{p+1}\M$ are supported on $X$ \cite[Eq.\,2.19]{BrHRPo20}. Since the sheaves $\M^{(p)}$ are coherent \cite[Prop.\,2.37]{BrHRPo20}, the statement is true for them by \cite[Exp.\,XII, Th\'eor\`eme 4.2]{SGA1} and one finishes by descending recursion.
\end{proof}

A procedure analogous to the   analytification of a superscheme allows one  to construct the complex supermanifold associated to a smooth analytic space. We recall the precise definition of differentiable supermanifold.

\begin{defin}\label{def:supermanifold} A differentiable supermanifold is a locally split locally ringed space $\Xcal=(X,\Ac_\Xcal)$ whose bosonic reduction is a differentiable manifold $(X,\Ac_X)$ (where $\Ac_X$ denotes the sheaf of real differentiable functions).
\end{defin}

Since we are working over $\C$, we normally consider the locally ringed superspace $\Xcal_\C=(X, \Ac_{\Xcal_\C})$ with $\Ac_{\Xcal_\C}=\Ac_\Xcal\otimes_\R\C$.

Now, if $\Xcal$ is a smooth analytic superspace of dimension $m|n$, one can associate to it a supermanifold of real dimension $2m|2n$; namely, one easily proves that the functor 
$$
\Ycal \mapsto \Hom(\Ycal_\C,\Xcal)
$$
(Homs of complex locally ringed spaces) on the category of differentiable supermanifolds is representable by a  differentiable supermanifold $\Xcal^{\infty}=(X, \Ac_{\Xcal^{\infty}})$ of real dimension $2m|2n$, so that one has 
\begin{equation}\label{eq:diff}
 \Hom(\Ycal_\C,\Xcal)\simeq  \Hom(\Ycal_C,\Xcal^{\infty}_\C)
\end{equation}
for every differentiable supermanifold $\Ycal$ (see \cite{HaWe87}). 

By construction there is a morphism of locally ringed spaces
$$
\gamma\colon \Xcal^\infty \to \Xcal
$$
which is the identity on $X$ and the inclusion $\Oc_\Xcal\hookrightarrow \Ac_{\Xcal\_C}$ on sheaves of superfunctions, where we have simply written $\Ac_{\Xcal_\C}:= \Ac_{\Xcal^{\infty}}\otimes_\R\C$. Moreover, if $\Jc$ (resp. $\Jc^\infty$) is the ideal of $\Oc_\Xcal$ (resp. $\Ac_{\Xcal^{\infty}}$) generated by the odd functions, then the corresponding sheaves $\Ec= \Jc/\Jc^2$, $\Ec^\infty=\Jc^{\infty}/\Jc^{\infty 2}$ are related by
$$
\Ec\otimes_{\Oc_X}\Ac_{\Xcal_\C} \simeq \Ec^\infty\otimes_\R\C\,.
$$
It follows that if $U$ is a coordinate neighbourhood for $\Xcal$ with local coordinates $\{z_i,\theta_J\}$, ($1\leq i\leq m$, $1\leq J\leq n$) where $(\theta_1,\dots,\theta_n)$ is a basis of the free $\Oc_X$-module $\rest\Ec U$, (that is, super holomorphic coordinates), then $(x_i,y_i,\eta_J,\varpi_J)$, where,
$$
x_i=\frac12(z_i+\bar z_i)\,,\quad y_i=\frac 12(z_i-\bar z_i)\,, \eta_J=\frac12(\theta_J+\overline{\theta_J})\,,\ \varpi_J=\frac12(\theta_J-\overline{\theta_J})\,.
$$
are super differentiable coordinates for $\Xcal^\infty$ on $U$.

Moreover, the superholomorphic functions on $U$ are characterized by the super Cauchy-Riemann equations \cite[Lemma 1]{HaWe87}:
\begin{prop}[Super Cauchy-Riemann equations]\label{prop:CReqs} Let $f\in \Ac_{\Xcal_\C}(U)$ be a complex superfunction of  $\Xcal^\infty$ on $U$. Then $f$ is super holomorphic, $f\in \Oc_\Xcal(U)$, if and only if 
$$
\frac {\partial f}{\partial \bar z_j}=\frac{\partial f}{\partial \bar \theta_J}=0\,, \quad 1\leq j\leq m\,,\ 1\leq J\leq n\,.
$$
\qed
\end{prop}

Let us denote by $\Omega_\Xcal$ the $\Oc_\Xcal$-module of holomorphic differentials and by $\Omega^p_\Xcal=\bigwedge_{\Oc_\Xcal}^p \Omega_\Xcal$ the $\Oc_\Xcal$-module of holomorphic $p$-forms. For smooth forms we use a different notation, writing $\Ac^1_{\Xcal_\C}$ for the $\Ac_{\Xcal_\C}$-module of differentiable complex 1-forms and $\Ac^p_{\Xcal_\C}=\bigwedge_{\Ac_{\Xcal_\C}}^p (\Ac^1_{\Xcal_\C})$.

\section[de Rham and Hodge cohomology for superschemes and\\\mbox{} analytic superspaces]{de Rham and Hodge cohomology for superschemes and analytic superspaces}\label{s:derham}

In this Section we recall and develop some basics of de Rham and Hodge cohomology in the super setting.
Starting from this Section, and throughout the rest of the paper, we assume that superschemes are locally of finite type over a field, usually the complex numbers, 
 and, as we said in the Introduction,  that all morphisms are locally of finite type.

\subsection{de Rham cohomology}

The definition of de Rham cohomology for superschemes is analogous to that  for ordinary schemes (see, for instance \cite[Chap.\,50]{Stacks}).
Let $f\colon \Xcal \to \Sc$ be a morphism of superschemes and $\Omega_{\Xcal/\Sc}$ the sheaf of relative differentials. The de Rham complex of $\Xcal/\Sc$ is the complex
$$
\Omega^\bullet_{\Xcal/\Sc}:= \Oc_\Xcal \xrightarrow{d} \Omega_{\Xcal/\Sc} \xrightarrow{d} \Omega^2_{\Xcal/\Sc} \xrightarrow{d} \dots
$$
where $\Omega^p_{\Xcal/\Sc} = \bigwedge_{\Oc_\Xcal} \Omega_{\Xcal/\Sc}$ and $d$ is the differential. This  is a bounded below complex of $f^{-1}\Oc_\Sc$-modules and its derived push-forward $R f_\ast \Omega^\bullet_{\Xcal/\Sc}$ is an object of the derived category $D(\mathfrak{Mod}^+(\Sc))$.
\begin{defin}\label{def:deRham} The de Rham cohomology groups of $f\colon \Xcal \to \Sc$ are the hypercohomogy $\Oc_\Sc(S)$-modules     of the complex $\Omega^\bullet_{\Xcal/\Sc}$,
\begin{align*}
	H^i_{dR}(\Xcal/\Sc)&= \Hs^i(X,\Omega^\bullet_{\Xcal/\Sc})\\
			   &:=H^i (R \Gamma (X, \Omega^\bullet_{\Xcal/\Sc}))\simeq H^i(R \Gamma (S, R f_\ast \Omega^\bullet_{\Xcal/\Sc}))\,.
\end{align*}
\end{defin}

We will mainly consider the absolute case, when $\Sc=\Spec \C$ is a single point.

Similar definitions can be given for analytic superspaces or differentiable supermanifolds. In the latter case, the super Poincar\'e Lemma (\cite[Thm.\,4.6]{Kost75}, \cite{BBH91}) implies that the de Rham cohomology groups are the real cohomology groups of the underlying differentiable manifold. We review the proof of the super Poincar\'e Lemma as well as some similar results in a more general context.

To begin, note that if $\Xcal$ is a superscheme there is a morphism $\Omega_\Xcal^\bullet \to \Omega_X^\bullet$ of complexes of vector spaces.
\begin{lemma}\label{lem:superPoincare}
If $\Xcal$ is locally split, $\Omega_\Xcal^\bullet \to \Omega_X^\bullet$ is a quasi-isomorphism; taking hypercohomolgy one then has an isomorphism
$$
H^i_{dR}(\Xcal) \iso H^i_{dR}(X)\,.
$$
The same statement holds for analytic superspaces and for differentiable supermanifolds
\end{lemma}
\begin{proof} The question is local, so that we can assume the $\Xcal$ is affine and split, so that $\Xcal=\SSpec \As$, $X=\Spec A$ and $\As\simeq \bigwedge_{A}E$ where $E=\langle \theta_1\dots,\theta_n\rangle$ is a free $A$-module. Then $\Omega_\As^\bullet$ decomposes as a tensor product of complexes $\Omega_\As^\bullet \simeq \Omega_A^\bullet\otimes_\C \bigwedge^\bullet_{\C[\theta_1\dots,\theta_n]}\langle d\theta_1\dots,d\theta_n\rangle$. In the second factor the $\theta_i$ anti-commute and the $d\theta_j$ commute, so that if we write $X_i=d\theta_i$, we see that the second factor is (up to a shift) the dual of the ordinary Koszul complex defined over the commutative polynomial ring $\C[X_1,\dots,X_n]$ by the free module generated by the $\theta_i's$. Thus it is acyclic, and the result follows by algebraic K\"unneth. The remaining cases are analogous.
\end{proof}

\begin{lemma}[Super Poincar\'e Lemma]\label{lem:superdR}
If $\Xcal$ is a differentiable supermanifold, for every   $i\geq 0$  there is an isomorphism
$$H^i(X,\R)\simeq H^i_{dR}(\Xcal)\simeq H^i(\Gamma(X,\Ac_\Xcal^\bullet)).$$\end{lemma}

\begin{proof} By the classical Poincar\'e Lemma, $\R \to \Ac^\bullet_X$ is a quasi-isomorphism, so that $\R$ and $\Ac^\bullet_\Xcal$ are quasi-isomorphic by Lemma \ref{lem:superPoincare}. Taking hypercohomolgy one gets $H^i(X,\R)\simeq H^i_{dR}(\Xcal)=H^i(R \Gamma(X,\Ac_\Xcal^\bullet))$. Moreover, the sheaves $\Ac_\Xcal^p$ are fine, and then acyclic, so that $H^i(R \Gamma(X,\Ac_\Xcal^\bullet))\simeq H^i(\Gamma(X,\Ac_\Xcal^\bullet))$. \end{proof}
This of course implies an isomorphism
$$H^i(X,\C) \simeq H^i(\Gamma(X,\Ac_\Xcal^\bullet\otimes_\R\C)).$$

\subsubsection{de Rham cohomology for smooth analytic superspaces}
Let us consider now the smooth analytic case. Let $\Xcal$ be an analytic superspace. The de Rham complex
$$
\Omega^\bullet_{\Xcal}:= \Oc_\Xcal \xrightarrow{d} \Omega_{\Xcal} \xrightarrow{d} \Omega^2_{\Xcal/} \xrightarrow{d} \dots
$$
is also know as the \emph{super holomorphic} de Rham complex. As in the non-super case, there is an inclusion of complexes of $\Omega^\bullet_{\Xcal}$ into the differentiable de Rham complex $\Ac^\bullet_{\Xcal_\C}$ of the associated complex differentiable supermanifold $\Xcal^\infty=(X,\Ac_{\Xcal_\C})$ (see Section \ref{ss:gaga}), that we describe now.

There is decomposition into $(p,q)$ types
$$
\Ac_{\Xcal_\C}^k=\bigoplus_{p+q=k}\Ac_{\Xcal_\C}^{p,q}
$$
which is homogeneous of degree 0 with respect to the $\Z_2$ grading.
The super exterior derivative decomposes as 
$d=\partial+\bar\partial$,
where $\partial$ maps $\Ac_\C^{p,q}$ into $\Ac_\C^{p+1,q}$ and $\bar\partial$ maps $\Ac_\C^{p,q}$ into $\Ac_\C^{p,q+1}$. Moreover, $d^2=0$ implies $\partial^2=\bar\partial^2=0$ and $\partial\bar\partial+\bar\partial\partial=0$, and then the differentiable de Rham  complex $\Ac^\bullet_{\Xcal_\C}$ is the simple complex associated to the de Rham bicomplex $\Ac_{\Xcal_\C}^{\bullet,\bullet}=(\Ac^{p,q}_{\Xcal_\C},\partial,\bar\partial)$. 

By Proposition \ref{prop:CReqs}, $\Omega_\Xcal^p$ is the kernel of $\Ac_{\Xcal_\C}^{p,0}\xrightarrow{\bar\partial} \Ac_{\Xcal_\C}^{p,1}$. Moreover, 
a super Dolbeault theorem also holds  \cite[Thm.\,2]{HaWe87}:
\begin{prop}[Super Dolbeault]\label{prop:superdolbeault} Let $\Xcal$ be a smooth analytic superspace. For every $p\geq 0$, the complex 
$$
0 \to \Omega_\Xcal^p\to \Ac_{\Xcal_\C}^{p,0}\xrightarrow{\bar\partial} \Ac_{\Xcal_\C}^{p,1} \xrightarrow{\bar\partial} \Ac_{\Xcal_\C}^{p,2}\xrightarrow{\bar\partial} \dots
$$
 is a resolution of the sheaf $\Omega_\Xcal^p$ of super holomorphic $p$-forms, called the super Dolbeault resolution.
\qed\end{prop}

The restriction of $\partial$ to $\Omega_\Xcal^p$ is the differential of the super holomorphic de Rahm complex. Then there is an inclusion of complexes
$\Omega_{\Xcal}^\bullet \hookrightarrow \Ac_{\Xcal_\C}^\bullet$.
By the standard theory of bicomplexes, one has:

\begin{prop}\label{prop:bicomplex} Let $\Xcal$ be a smooth analytic superspace. 
\begin{enumerate} \item
The inclusion $\Omega_{\Xcal}^\bullet \hookrightarrow \Ac_{\Xcal_\C}^\bullet$ is a quasi-isomorphism. 
\item
There are isomorphisms $H^i_{dR}(\Xcal)\simeq H^i_{dR}(\Xcal^\infty)\simeq H^i (X,\C)$ for every $i\geq 0$.
\end{enumerate}
\end{prop}

\begin{proof} (1) follows directly from \cite[Thm.\,4.8.1]{Go73} and (2)  by taking hypercohomology and applying the super de Rham Lemma \ref{lem:superdR}.
\end{proof}

Thus the de Rham cohomology for a smooth analytic superspace does not provide new invariants.

\subsection{Hodge cohomology}

Hodge cohomology for superschemes (or other geometric super structures) can be defined   the way one   expects.
\begin{defin}\label{def:hodge} Let $\Xcal$ be either a superscheme over a field $k$ or a analytic superspace. The Hodge cohomology groups of $\Xcal$ are the groups
$$
H^{p,q}(\Xcal)=H^q(X,\Omega_\Xcal^p)\,.
$$
They are naturally $\Z_2$-graded vector spaces over the respective base fields. 
\end{defin} 

The Hodge cohomology groups of a proper complex superscheme coincide with the Hodge groups of the associated analytic superspace by the super GAGA theorems:
\begin{prop} Let $\Xcal$ be a proper complex  superscheme and $\Xcal^{an}$ its analytification (Definition \ref{def:analytization}). One has
$$
H^{p,q}(\Xcal)\simeq H^{p,q}(\Xcal^{an})\,.
$$
\end{prop}
\begin{proof} It follows from Proposition \ref{prop:GAGA} as $\Omega_\Xcal^{an}\simeq \Omega_{\Xcal^{an}}$.
\end{proof}

\subsubsection{Hodge cohomology for smooth analytic superspaces}

Let $\Xcal$ be an analytic superspace. As in the classical case, the Hodge cohomology groups of $\Xcal$ can be described in terms of $\bar\partial$ acting on the complex of global sections of the Dolbeault resolution.
\begin{prop}\label{prop:hodge}
One has
$$
H^{p,q}(\Xcal)\simeq H^q(\Gamma(X,\Ac_{\Xcal_\C}^{p,\bullet}),\bar\partial)\,.
$$
\end{prop}
\begin{proof} It follows from the Super Dolbeault Proposition \ref{prop:superdolbeault} since the sheaves $\Ac_{\Xcal_\C}^{p,q}$ are fine.
\end{proof}

If we consider the bicomplex $\Ac_{\Xcal_\C}^{\bullet,\bullet}(X)=(\Ac^{p,q}_{\Xcal_\C}(X),\partial,\bar\partial)$ of global sections of the de Rham bicomplex, Proposition \ref{prop:superdolbeault} implies that its cohomology with respect to $\bar\partial$ is actually the superholomorphic complex of $\Xcal$. One then has
\begin{prop}[Fr\"ohlicher spectral sequence]\label{prop:frolicher}
There is a convergent spectral sequence  
$$
E_2^{p,q}=H^{p,q}(\Xcal) \implies E_\infty^{p+q}=H^{p+q}_{dR}(\Xcal)\simeq H^{p+q}(X,\C)\,.
$$
that relates Hodge and de Rham cohomologies.
\qed
\end{prop}

\begin{remark} In the classical case, the Hodge decomposition $H^{n}(X,\C)=\oplus_{p+q=n} H^{p,q}(X)$ implies that the Fr\"ohlicher spectral sequence degenerates at page one. In the super setting, such a decomposition may fail to exist. Actually, a Hodge decomposition would imply that all the cohomology vector superspaces $H^{p,q}(\Xcal)$ are even, which is false in most cases. Take, for instance, a supercurve $\Xcal=(X,\Oc_{\Xcal}=\Oc_X\oplus\Lcl)$, where $X$ is a smooth proper curve of genus $g$ and $\Lcl$ is a line bundle on $X$; then $H^{0,0}(\Xcal)=H^0(X, \Oc_{\Xcal})=\C\oplus\Pi H^0(X,\Lcl)$, which is different from $H^0(X^{an},\C)$ if $H^0(X,\Lcl)\neq 0$.

Even the weaker statement $H^{n}(X,\C)=\oplus_{p+q=n} H^{p,q}(\Xcal)_+$ is in general false: for a supercurve as above, one can decompose $\Omega_\Xcal$ \emph{as an $\Oc_X$-module}, in the form $\Omega_{\Xcal}\simeq (\Omega_X\oplus\Lcl^2)\oplus\Pi (\Lcl\otimes \Omega_X\oplus\Lcl)$, so that $H^{0,1}(\Xcal)_+\simeq\C^g\oplus H^1(X,\Lcl)$ and $H^{1,0}(\Xcal)_+\simeq\C^g\oplus H^0(X,\Lcl^2)$; thus, $H^1(X^{an},\C)\neq H^{0,1}(\Xcal)_+\oplus H^{1,0}(\Xcal)_+$ whenever either $H^1(X,\Lcl)$ or $H^0(X,\Lcl^2)$ is nonzero (for instance, if $\Lcl=\Oc_X$).

Another proof of the non-degenerateness of the Fr\"ohlicher spectral sequence at page one for a supercurve with $g=1$ and $\Lcl=\Oc_X$ is \cite[Example 5.13]{CaNoRe20}.
\end{remark}

\section{Supercycles}\label{s:supercycles}

In this Section we develop a theory of supercycles for superschemes, adapting to the ``super'' setting the classical theory (see, for instance, \cite{Fu98, Stacks}).

\subsection{Preliminary algebraic results} We shall need some preliminary results about supercommutative algebra. In this part we do not 
need to assume that the superrings we consider are algebras over a field.

\subsubsection{Artin superrings and $Z_2$-graded modules}
We follow mainly \cite{We09}. The results given here without a proof can be found there.
\begin{lemma}\label{lem:artin}
A superring $\As$ is artinian if and only if the bosonic reduction $A$ is an artinian commutative ring.
\end{lemma}
\begin{proof} It is clear that if $\As$ is artinian, then $A$ is artinian as well. For the converse, note that for some $n$ there is a filtration $\As\supset J_\As\supseteq J_\As^2\supseteq \dots \supseteq J_\As^n\supset J_\As^{n+1}=0$. The successive quotients are finitely generated $A$-modules,  so that they are artinian $A$-modules because $A$ is artinian. Then they are artinian $\As$-modules as they are annihilated by $J_\As$, so that all the terms of the filtration are artinian $\As$-modules.
\end{proof}

As in \cite{We09} a composition series in a $\Z_2$-graded module $M$ over a superring $\As$ is a chain
$$
0=M_0\subset M_1\subset\dots
$$
of submodules of $M$ such that the successive quotients (or factors) $M_i/M_{i-1}$ are simple $\As$-modules. Then $M_i/M_{i-1}$ is either \emph{even}, that is, it is isomorphic to $\As/\af$ for an ideal $\af$ of $\As$, or \emph{odd}, i.e., it is isomorphic to $\Pi\As/\bfr$ for an ideal $\bfr$ of $\As$. The module $M$ is said to be of \emph{finite length} if it has a finite composition series.

Let us consider the commutative ring $\Z^2=\Z\times\Z$ with the product given by $(m,n)\cdot (m',n')=(mm'+nn', mn'+m'n)$.
\begin{defin}\label{def:length} The (super) length of a $\Z_2$-graded $\As$-module $M$ of finite length is
$$
\ell_\As(M)=(\ell_\As(M)_+,\ell_\As(M)_- )\in \Z^2\,,
$$
where $\ell_\As(M)_+$ (resp. $\ell_\As(M)_-$) is the number of even (resp. odd) successive quotients of a finite composition series
$$
0=M_0\subset M_1\subset\dots\subset M_n=M\,.
$$
One easily sees that $\ell_\As(M)$ so defined is independent of the choice of a composition series of $M$.
\end{defin}
\begin{remark} The length defined in \cite[3.4]{We09} is an integer number, actually the sum of $\ell_\As(M)_+$ and $\ell_\As(M)_-$, i.e., the length of any of the composition series of $M$.
\end{remark}

\begin{lemma}\label{lem:lenght} 
\begin{enumerate}
\item
Let $f\colon \As\to \Bs$ be a flat local morphism of local superrings. For every $\Z_2$-graded $\As$-module $M$ one has
$$
\ell_\Bs(M\otimes_\As \Bs)= \ell_\As(M)\cdot \ell_\Bs(\Bs/\mf_\As \Bs)
$$
with the product of $\Z^2$. In particular, if $\Bs/\mf_\As \Bs$ is of finite length over $\Bs$, then $M$ is of finite length if and only if $M\otimes_\As\Bs$ is of finite length.
\item If $\As$ and $\Bs$ are artinian, then
$$
\ell_\Bs(\Bs)= \ell_\As(\As)\cdot \ell_\Bs(\Bs/\mf_\As\Bs)\,.
$$
\end{enumerate}
\end{lemma}
\begin{proof} (1) 
 Since $f$ is flat, its bosonic reduction $f_\bos\colon A\to B$ is flat as well; as the induced map $\SSpec \Bs \to \SSpec\As$ is surjective onto the maximal spectrum, $f_\bos $ is faithfully flat. One finishes easily by using Lemma 2.10 of \cite{BrHR21}.\footnote{In Definition \ref{def:proper} we have defined a faithfully flat morphism of superschemes as a flat surjective one. Here we have proved that this implies that the local rings of the target are faithfully flat algebras over the local rings of the source in the algebraic meaning.}

Now, if 
$0=M_0\subset M_1\subset\dots\subset M_n=M$ is a finite chain in $M$, one gets a finite chain $0=M_0\otimes_\As\Bs\subset M_1\otimes_\As\Bs\subset\dots\subset M_n\otimes_\As\Bs=M\otimes_\As\Bs$ in $M\otimes_\As\Bs$. If the $\As$-module $M$ is not of finite length, the $\Bs$-module $M\otimes_\As\Bs$ is not of finite length either. If $M$ is of finite length $\ell=(\ell_+,\ell_-)$ and $0=M_0\subset M_1\subset\dots\subset M_n=M$ is a finite composition series, the chain $0=M_0\otimes_\As\Bs\subset M_1\otimes_\As\Bs\subset\dots\subset M_n\otimes_\As\Bs=M\otimes_\As\Bs$ has $\ell_+$ even factors isomorphic to $\Bs/\mf_\As\Bs$ and $\ell_-$ odd factors isomorphic to $\Pi\,\Bs/\mf_\As\Bs$. 
Each even factor has a composition series with $\bar\ell_+$ factors isomorphic to $\Bs/\mf_\Bs$ and $\bar\ell_-$ factors isomorphic to $\Pi\,\Bs/\mf_\Bs$, where $\bar\ell=(\bar\ell_+,\bar\ell_-)$ is the length of $\Bs/\mf_\As\Bs$ as a $\Bs$-module. It follows that $M\otimes_\As\Bs$ has a composition series with $\ell_+\bar\ell_++\ell_-\bar\ell_-$ even factors and $\ell_+\bar\ell_-+\ell_-\bar\ell_+$ odd factors, which finishes the proof.

(2) follows from (1) taking $M=\As$.
\end{proof}

\subsubsection{Order functions}

Let $\Bs$ be a superdomain of even dimension 1. If $b\in\Bs$ is even and not nilpotent, the quotient $\Bs/(b)$ is an Artin superring. We can then define the order of $b$ as the (super) length of $\Bs/(b)$ as a module over itself:
$$
\ord_{\Bs}(b)=\ell(\Bs/(b))\in \Z^2\,.
$$
As in the non-super situation, the order transforms products into sums, that is, $\ord_{\Bs}(b b')=\ord_{\Bs}(b)+\ord_{\Bs}(b')$.

The order can be extended to rational superfunctions.
\begin{defin}\label{def:order} If $g\in \K(\Bs)^\ast$ an even rational superfunction. The order of $g$ is
$$
\ord_{\Bs}(g)=\ord_{\Bs}(b_1)-\ord_{\Bs}(b_0)\in \Z^2\,,
$$
where $g=b_1/b_0$ with $b_0,b_1\in\Bs-J_{\Bs}$. The order is independent on the representation of $g$ as a quotient of two elements of $\Bs$ because the order transforms products in $\Bs$ into sums.
\end{defin}

By its very definition, one   has
\begin{equation}\label{eq:order}
\ord_{\Bs}(g g')=\ord_{\Bs}(g)+\ord_{\Bs}(g')\,,
\end{equation}
for even rational superfunctions.

\subsubsection{Superlattices and distance}\label{sss:lattices}

Let $\Bs$ be a superdomain of even dimension 1 and $E$ a free module of rank $p|q$ over the superfield $\K(\Bs)$. 

\begin{defin}
A superlattice in $E$ is a graded finite $\Bs$-submodule $M\subset E$ such that $M\otimes_\Bs\K(\Bs)\iso E$.
\end{defin}
The analogous super version of \cite[Lemma 10.121.4]{Stacks} holds. In particular, given two superlattices $M$, $M'$ in $E$, the $\Bs$-modules $M/M\cap M'$ and $M'/M\cap M'$ have finite length.

\begin{defin}\label{def:distance} The distance between $M$ and $M'$ is
$$
d(M,M')= \ell_\Bs(M/M\cap M')-\ell_\Bs(M'/M\cap M') \in \Z^2\,.
$$
\end{defin}

\begin{lemma}\label{lem:distandord} Let $\phi\colon E \to E$ be an (even) isomorphism of $\K(\Bs)$-modules. For any superlattice $M\subset E$ one has
$$
d(M,\phi(M))=\ord_\Bs(\ber(\phi))\,,
$$
where $\ber(\phi)\in \Bs$ is the berezinian of the isomorphism $\phi$
\end{lemma}
\begin{proof}It is analogous to the proof of \cite[Lemma 10.121.4]{Stacks}. One first proves as there that $d(M,\phi(M))$ is independent of the superlattice $M$ and that the first member of the formula is additive on $\phi$, that is,
$$
d(M,\psi\circ\phi (M))= d(M,\psi(M))+d(M,\phi(M))\,,
$$
if $\psi$ is another automorphisms of $E$. The second member is also additive in $\phi$, due to $\ber(\psi\circ\phi)=\ber(\psi)\cdot\ber(\phi)$ and to the additivity of the degree (Equation \eqref{eq:order}). 

Choose a free basis of the $\K(\Bs)$-module $E$ and represent the automorphisms as   matrices in the superlinear group $GL_\Bs(p|q)$. The, it is enough to prove the formula for supermatrices that generate the group. The block matrices
$$
\begin{pmatrix}
I& 0\\C_3 &I
\end{pmatrix}\,, \quad \begin{pmatrix}
C_1& 0\\0 &C_4
\end{pmatrix}
\quad \begin{pmatrix}
I& C_2\\0 &I
\end{pmatrix}
$$
where $C_2$ is elementary and $C_1$, $C_4$ are invertible, generate $GL_\Bs(p|q)$ (see the proof of \cite[\S 3, Thm.\,5]{Ma99}). One easily see that we can also assume that $C_3$ is elementary. For any of those matrices, the formula of the statement is proven, along the same lines as the proof of \cite[Lemma 10.121.4]{Stacks}, by a standard computation.
\end{proof}

\subsubsection{Superdomains of even dimension 1}
Let $\As$, $\Bs$ be  superdomains of even dimension 1, with $\As$ local, $\As \to \Bs$ a finite morphism and $f\colon \Xcal=\SSpec\Bs \to \Ycal=\SSpec\As$ the induced superscheme morphism. The superdomain $\Bs$ is semilocal and its maximal ideals correspond to the (closed) points $\{x_1,\dots,x_s\}$ of the fibre $f^{-1}(y)$ over the (unique) closed point $y\in \Ycal$.
\begin{prop}\label{prop:preparation} Assume that $\K(\Bs)$ is free as a $\K(\As)$-module. If $g\in \K(\Bs)$ is an invertible rational superfunction of $\Xcal$, one has
$$
\ord_y(\ber(g))=\sum_{i=1}^s \dim_{\kappa(y)}\kappa(x_i)\cdot \ord_{x_i}(g)\,,
$$
where we have written $\ord_y(\ )=\ord_{\As}(\ )$, $\ord_{x_i}(\ )= \ord_{\Bs_{x_i}}(\ )$ for simplicity and $\kappa(y)$, $\kappa(x_i)$ are the residue fields.
\end{prop}
\begin{proof} Since  the order  is additive and      the berezinian is multiplicative, it is enough to prove the statement for $g\in\Bs$. By Lemma \ref{lem:distandord} one   has
$$
\ord_y(\ber(g))=d(\Bs, g\cdot\Bs)=\ell_{\As}(\Bs/g\cdot\Bs)\,.
$$
Moreover, $\Bs/g\cdot\Bs\simeq \bigoplus_{i=1}^s \Bs_{x_i}/g\cdot \Bs_{x_i}$ and then 
$$
\ell_{\As}(\Bs/g\cdot\Bs)=\sum_{i=1}^s \ell_{\As}(\Bs_{x_i}/g\cdot \Bs_{x_i})=\sum_{i=1}^s \dim_{\kappa(y)}\kappa(x_i)  \ell_{\Bs_{x_i}}(\Bs_{x_i}/g\cdot \Bs_{x_i})\,,
$$
thus finishing the proof.
\end{proof}

\subsection{Supercycles on superschemes}

Henceforth all the superschemes  will be locally of finite type over an algebraically closed field, and all morphisms of superschemes will be locally of finite type.

\begin{defin} Let $\Xcal$ be a superscheme of even dimension $m$. 
An $h$-supercycle of $\Xcal$ is a finite sum
$$
\alpha=\sum_i (m_i, n_i) [Y_i]
$$
where $(m_i,n_i)\in\Z^2$ and the $Y_i$ are subvarieties of $X$ of dimension $h$. The set of   $h$-supercycles is a free $\Z_2$-graded module over $\Z^2$.
The group of supercycles of $\Xcal$ is the direct sum
$$
Z_\ast(\Xcal)=\bigoplus_{h=0}^m Z_h(\Xcal)\,.
$$
It is a bigraded $\Z^2$-module.
\end{defin}

Restriction to the bosonic underlying scheme $i\colon X \hookrightarrow \Xcal $ gives a morphism of graded $\Z^2$-modules
\begin{equation}\label{eq:bosred}
i^\ast\colon Z_h(\Xcal) \to Z_h (X)\otimes_\Z\Z^2\,.
\end{equation}

If $j\colon \Zc\hookrightarrow \Xcal$ is a closed immersion of superschemes, there is a natural injective morphism of $\Z^2$-modules
$$
j_\ast\colon Z_h(\Zc)\hookrightarrow Z_h(\Xcal)\,.
$$

Let $\Xcal$ be a superscheme and $j\colon\Zc\hookrightarrow\Xcal$ a closed sub-superscheme of pure even dimension $h$. Let $Z_1,\dots,Z_s$ be the irreducible components of $Z$. The local ring $\Oc_{\Zc,z_i}$ of $\Zc$ at the generic point $z_i$ of $Z_i$ is an Artin superring. Its (super) length $m_{Z_i}(\Zc):=\ell(\Oc_{\Zc,z_i})\in\Z^2$ is called the \emph{geometric supermultiplicity} of $\Zc$ at $Z_i$.
\begin{defin}\label{def:fundsupercycle} The supercycle of $\Zc$ is the supercycle
$$
[\Zc]=\sum_i m_{Z_i}(\Zc) [Z_i]\in Z_h(\Zc)\,.
$$
The supercycle of $\Zc$ in $\Xcal$ is the image $j_\ast[\Zc]$ of $[\Zc]$ by $j_\ast\colon Z_h(\Zc)\hookrightarrow Z_h(\Xcal)$.
\end{defin}

When $\Zc$ is a sub-supervariety its associated supercycle is
$$
[\Zc]=m_{Z}(\Zc) [Z]=\ell(\Oc_{\Zc,z})[Z]\,,
$$
where $z$ is the generic point.

\subsection{Flat pullback of supercycles}

Let $f\colon \Xcal\to \Ycal$ be a flat morphism of superschemes of relative even dimension $m$. 
If $Z$ is a subvariety of dimension $h$ of the scheme $Y$, the pullback $f^{-1}(Z)$ is a closed sub-superscheme of $\Xcal$ of even dimension $h+m$ whose bosonic reduction is $f_{\bos}^{-1}(Z)\hookrightarrow X$. Denote by $j\colon f^{-1}(Z)\hookrightarrow \Xcal$ the corresponding closed immersion.

\begin{defin}\label{def:pullback} The pullback of $[Z]$ is the supercycle given by
$$
f^\ast[Z]=j_\ast [f^{-1}(Z)]\,,
$$
where $[f^{-1}(Z)]$ is the supercycle given by Definition \ref{def:fundsupercycle}. This extends to a morphism
$$
f^\ast\colon Z_h(\Ycal)\to Z_{h+m}(\Xcal)
$$
of $\Z^2$-modules.
\end{defin}

A nice property of the pullback of supercycles is the following.
\begin{prop}\label{prop:pullback} If $\Zc$ is a closed sub-superscheme of $\Ycal$ of pure even dimension $h$, then
$$
f^\ast [\Zc]= [f^{-1}(\Zc)]\,.
$$
\end{prop}
\begin{proof} Since the base-change morphism $\Xcal_\Zc=f^{-1}(\Zc) \to \Zc$ is flat, we can assume that $\Zc=\Ycal$. Let $X_i$ be the irreducible components  of $\Xcal$. If $\xi_i$ is the generic point of $X_i$ and $\xi$ is the generic point of $Y$, the superring morphism $\As=\Oc_{\Ycal,\xi} \to \Bs=\Oc_{\Xcal,\xi}$ is flat, because $f$ is flat. The coefficients of $X_i$ in $[f^{-1}(\Ycal)]=[\Xcal]$ and $f^\ast[\Ycal]$ are $\ell_\Bs(\Bs)$ and $\ell_\As(\As)\cdot \ell_\As(\Bs/\mf_\As \Bs)$, and they are equal by Lemma \ref{lem:lenght}.
\end{proof}
The following result is now clear.
\begin{corol}\label{cor:pullback} The pullback of supercycles is functorial, that is, if $f\colon \Xcal\to \Ycal$, $g\colon \Wc \to \Xcal$ are flat morphisms of superschemes of relative even dimensions $m$ and $m'$, respectively, then
$$
(f\circ g)^\ast = g^\ast\circ f^\ast
$$
as morphisms  of $\Z^2$-modules from $Z_h(\Ycal)$ to $Z_{h+m+m'}(\Wc)$.
\qed
\end{corol}
 \subsection{Superdivisors and rational equivalence}
 
Let $\Wc$ be a supervariety  of  even dimension $h+1$. 
For every 1-codimensional subvariety $V$ of $W$, the local superring $\Bs=\Oc_{\Wc,\xi}$ of $\Wc$ at the generic point $\xi$ of $V$, is a superdomain of  even dimension 1. We then have an associated order function (Definition \ref{def:order}).
$$
\ord_V(f)=\ord_\Bs(f)\in\Z^2\,.
$$
\begin{defin}\label{def:div}
Let $g\in \K(\Wc)^\ast$ be an even rational superfunction.
The supercycle $\divi(g)\in Z_h(\Wc)$ associated to $f$ is
$$
\divi(g)=\sum \ord_V(g)[V]\,.
$$
The sum runs over all   1-codimensional subvarieties $V$ of $W$; the formula makes sense because $\ord_V(g)\neq 0$ only for a finite number of $V$'s.
\end{defin}

We now define rational equivalence for supercycles.

\begin{defin}\label{def:ratequiv2} Let $\Xcal$ be a superscheme. A supercycle $\alpha\in Z_h(\Xcal)$ is rationally equivalent to zero if there exist a finite number $i=0,\dots,t$ of sub-supervarieties $\delta_i\colon\Wc_i\hookrightarrow\Xcal$  of even dimension $h+1$ and pure odd dimension  $s=0,1$ (Definition \ref{def:dim}) of $\Xcal$ and nonzero rational even superfunctions $g_i\in \K(\Wc_i)^\ast$ such that
$$
\alpha=\sum_{i=0}^t \delta_{i\ast}\divi(g_i)\,.
$$
\end{defin}

\subsection{Proper push-forward of supercycles}

The push-forward of supercycles is defined in a similar way to the push-forward of cycles \cite{Fu98}.
Let $f\colon \Xcal\to\Ycal$ be a proper morphism of superschemes. For every   subvariety $Z$ of $X$, the scheme-theoretic image $f(Z)$ is a subvariety of $Y$. There is an extension $K(f(Z))\hookrightarrow K(Z)$ between the corresponding fields of algebraic functions, which is finite when $Z$ and $f(Z)$ have  the same dimension. One sets
$$
\deg(Z/f(Z))=\begin{cases} \dim_{K(f(Z))}K(Z) & \text{if $\dim Z=\dim f(Z)$} \\ 0 & \text{otherwise}\,.
\end{cases}
$$ 
\begin{defin}\label{def:push} The push-forward of the supercycle $[Z]$ of $\Xcal$ is the supercycle of $\Ycal$ given by
$$
 f_\ast([Z])= \deg(Z/f(Z)) [f(Z)]\,.
$$
This extends to a morphism of $\Z_2$-graded $\Z^2$-modules
$$
f_\ast \colon Z_h(\Xcal) \to Z_h(\Ycal)\,.
$$
\end{defin}

The push-forward of supercycles is compatible with composition, that is, if $f\colon\Xcal \to \Ycal$, $h\colon \Ycal \to \Zc$ are proper morphisms of superschemes, then 
$$
(h\circ f)_\ast=h_\ast\circ f_\ast
$$
as morphisms from $Z_h(\Xcal)$ to $Z_h(\Zc)$.

Since $\divi (g^{-1})=-\divi(g)$, the $h$-supercycles rationally equivalent to zero form a graded $\Z^2$-submodule $W_h(\Xcal)$ of $Z_h(\Xcal)$.
\begin{defin}\label{def:cyclesgr}
The $\Z^2$-module of $h$-supercycles modulo rational equivalence is the quotient
$$
A_h(\Xcal)= Z_h(\Xcal)/W_h(\Xcal)\,.
$$
\end{defin}

\begin{prop}\label{lem:basechange} Let 
$$
\xymatrix{
\Xcal'\ar[r]^{\phi'}\ar[d]^{f'} & \Xcal \ar[d]^f 
\\
\Ycal'\ar[r]^\phi & \Ycal
}
$$ 
be a cartesian diagram of superscheme morphisms where $f$ is proper and $\phi$ is flat. Then $f'$ is proper, $\phi'$ is flat and for every supercycle $\alpha$ in $\Xcal$ one as
$$
f'_\ast \phi^{'\ast}(\alpha)= \phi^\ast f_\ast(\alpha)\,.
$$
\end{prop}
\begin{proof}  The proof of the similar statement for schemes given in \cite[Prop.\,1.7]{Fu98}  can be adapted straightforwardly (using
  Lemma \ref{lem:lenght} in this paper instead of  \cite[Lemma A.1.3]{Fu98}). \end{proof}

\begin{lemma}\label{lem:degree}
Let $f\colon \Xcal \to \Ycal$ be a finite morphism of superschemes. Assume that $f$ is locally free of rank $d=d_+|d_-$, that is, that $f_\ast\Oc_\Xcal$ is a locally free finitely generated $\Z_2$-graded $\Oc_\Ycal$-module of rank $d$. For every supercycle $\alpha$ in $\Ycal$ one has \marginnote{Aggiunto un puntino}
$$
f_\ast f^\ast \alpha= d\cdot\alpha\,.
$$
\end{lemma}
\begin{proof} For every point $\xi\in Y$, if $x_1,\dots,x_s$ are the points of $f^{-1}(\xi)$, the localized superring $\Bs=\Oc_{\Xcal,\xi}$ is a free module of rank $d$ over the local superring $\As=\Oc_{\Ycal,\xi}$. Then $\Bs/\mf_\As\Bs$ is a super vector space of rank $d$ over $\kappa(\xi)=\As/\mf_\As$.
One has a primary decomposition $\mf_\As\Bs=\mf_{x_i}^{a_1}\cdot\dots\cdot\mf_{x_s}^{a_s}$, and then $\Bs/\mf_\As\Bs\simeq \bigoplus_i \Bs/\mf_{x_i}^{a_i}$ so that
\begin{equation}\label{eq:degree}
d= \sum_i \ell_\Bs(\Bs/\mf_{x_i}^{a_i}) \cdot \dim_{\kappa(\xi)}\kappa(x_i)\,.
\end{equation}
Now we proceed to the proof. One can assume that $\alpha=[Z]$ for a subvariety  $Z$ of $Y$. If $\xi$ is the generic point of $Z$, the points $x_i$  in the fibre over it  are the generic points of the irreducible components of $f^{-1}(Z)$, and then $[f^{-1}(Z)]=\sum_i \ell (\Oc_{f^{-1}(Z),x_i}) [Z_i]= \sum_i \ell_\Bs(\Bs/\mf_{x_i}^{a_i}) [Z_i]$. On the other hand, $f_\ast([Z_i])= \dim_{K(Z)}K(Z_i)\cdot [Z]= \dim_{\kappa(\xi)}\kappa(x_i) \cdot [Z]$. The result follows from Equation \eqref{eq:degree}.
\end{proof}

\subsubsection{ Divisors, zero loci  and loci of poles}

Let $f\colon \Xcal \to \Sc$ be a (locally of finite type) morphism of superschemes where $\Xcal$ is a supervariety, and let $g\in \K(\Xcal)^\ast$ be an even rational superfunction. We   prove that the divisor of $g$ is the difference between the supercycles of its locus of zeroes and   its locus of poles. There is an open sub-superscheme $\U$ of $\Xcal$ such that $g$ is an even section of $\Oc_\Xcal^\ast$ on $U$, and then $g$ induces a superscheme morphism $g\colon \U\to \Ps_{\Sc}^1=\Ps_{\Sc}^{1|0}$.
Let $\Ycal$ be the superscheme-theoretic image of the graph of $(1, g)\colon \U\to \Xcal\times_\Sc\Ps_{\Sc}^1$. The projection $p\colon \Ycal\to\Xcal$ is proper and induces an isomorphism $p^{-1}(\U)\iso \U$. We have a commutative diagram
\begin{equation}\label{eq:ratfunc}
\xymatrix{
\Ycal \ar[r]^p\ar[d]^q & \Xcal \ar[d]^f \\
\Ps_{\Sc}^1\ar[r]^\pi & \Sc
}
\end{equation}
where $\pi\colon \Ps_{\Sc}^1\to\Sc$ is the natural projection. \marginnote{Aggiunti due 1}
If we write $\Ps_\Sc^1=\SProj \Oc_\Sc[t_0,t_1]$ and denote by $\Sc_0$, $\Sc_\infty$ the homogeneous zeros of $t_1$ and $t_0$, respectively, proceeding as in \cite[Lemma 42.18.2]{Stacks} one has:
\begin{lemma} \label{lem:ratfunct}
\begin{enumerate} 
\item If $q\colon \Ycal \to \Ps_\Sc^1$ is the morphism induced by the projection onto the second factor, the pullbacks $\Ycal_0=q^{-1}\Sc_0$, $\Ycal_\infty=q^{-1}\Sc_\infty$ are Cartier positive superdivisors (i.e., closed sub-superschemes whose ideals are line bundles of rank $1|0$).
\item If we regard $g$ as a rational superfunction on $\Ycal$, its divisor in $\Ycal$ is 
$$
\divi_\Ycal(g)=j_{0\ast}[\Ycal_0]-j_{\infty\ast}[\Ycal_\infty]\,,
$$ 
 where the supercycles of $[\Ycal_0]$, $[\Ycal_\infty]$ are as in Definition \ref{def:fundsupercycle} and $j_0\colon \Ycal_0\hookrightarrow \Ycal$ and $j_\infty\colon \Ycal_\infty\hookrightarrow\Ycal$ are the immersions. \marginnote{defined $\to$ as}
\item If $\Xcal_0=p(\Ycal_0)$, $\Xcal_\infty=p(\Ycal_\infty)$ (superscheme-theoretic images), the divisor of $g\in \K(\Xcal)^\ast$  in $\Xcal$ is
$$
\divi_\Xcal(g)=j_{0\ast}[\Xcal_0]-j_{\infty\ast}[\Xcal_\infty]\,,
$$
where we denote also by $j_0$ and $j_\infty$ the corresponding immersions.
\end{enumerate}
\qed
\end{lemma}

\subsubsection{Compatibility between rational equivalence and push-forward}

Our next task is to prove that the push-forward of supercycles  is compatible with rational equivalence. The proof is inspired by the one given in \cite[Lemma 42.20.3]{Stacks} for ordinary cycles. Again, we need some preliminary results. 

\begin{lemma}\label{lem:funcdiagram} Let $\Wc$ be a supervariety of even dimension 1 and pure odd dimension 1 and $f\colon\Wc \to \SSpec \K$  a proper morphism where $\K$ is a superfield of odd dimension 0 or 1. If $g\in\K(\Wc)^\ast$ is an even rational superfunction, then $f_\ast[\divi(g)]=0$. \marginnote{corretto g in f}
\end{lemma}
\begin{proof} We have a commutative diagram
$$
\xymatrix{
\Ycal \ar[r]^p\ar[d]^q & \Xcal \ar[d]^f \\
\Ps_{\K}^1\ar[r]^(.4)\pi & \SSpec \K
} 
$$
(see Equation \eqref{eq:ratfunc}). Since $p$ is an isomorphism on an open sub-superscheme of $\Ycal$, the latter has  even dimension    1 and  pure odd dimension 1, that is, $\Oc_\Ycal=\Oc_Y\oplus \Jc_\Ycal$ where $\Jc$ is a $\Oc_Y$-module of generic rank 1. Moreover, we can choose the open $U$ involved in  the construction of the diagram Equation \eqref{eq:ratfunc} such that $\Jc$ is actually a torsion-free rank 1 $\Oc_Y$-module. Moreover $\K(\Ycal)=K(Y)\oplus \Jc_\xi$, $\xi$ being the generic point of $Y$. Then, $g\in K(Y)$ because it is even. This implies that $q\colon \Ycal \to \Ps_{\K}^1$ factors as the composition of a proper morphism $\tilde q\colon \Ycal\to \Ps_K^1$ and the closed immersion $i\colon \Ps_K^1\hookrightarrow \Ps_\K^1$ of the ordinary projective line over $K$ into $\Ps_\K^1$. Here $K$ is the bosonic reduction of $\K$.

By Lemma \ref{lem:ratfunct}, $\divi_\Xcal(g)=p_\ast\divi_\Ycal(g)$, and then we have to prove that $0=f_\ast(p_\ast\divi_\Ycal(g))=\pi_\ast(q_\ast\divi_\Ycal (g))$.
If $q(Y)$ is a closed point of $\Ps_\K$, then $q_\ast\divi_\Ycal(g)=0$ and the statement follows. Otherwise, $q$ is dominant. 
Let us see that $\tilde q$ is flat in this case. Since it is proper with finite fibres, $\tilde q$ is finite by Zariski's Main Theorem (Proposition \ref{prop:zariski}). Then, $\tilde q_\ast\Oc_\Ycal=\tilde q_\ast \Oc_Y\oplus \tilde q_\ast\Jc_\Ycal$ is the sum of two torsion-free modules over $\Oc_{\Ps_K}$. Since $\Ps_K^1$  is smooth, both $\tilde q_\ast \Oc_Y$ and $\tilde q_\ast\Jc_\Ycal$ are flat (actually, locally free).

Let us denote by $s_0$ and $s_\infty$ the zeros of $t_1$ and $t_0$ respectively in $\Ps_K^1=\Proj K[t_0,t_1]$. Since $\K$ has odd dimension 0 or 1, one has $\K=K$ or $\K=K[\theta]$; in either case $\Ps_\K^1$ is projected, that is, there is a projection $\rho\colon \Ps_\K^1\to \Ps_K^1$ such that $\rho\circ i=\Id$. Moreover, in the notation of Lemma \ref{lem:ratfunct}, one has $\Sc_0=\rho^{-1}(s_0)$ and $\Sc_\infty=\rho^{-1}(s_\infty)$.

Taking into account this property and Lemma \ref{lem:ratfunct}, one has
$\divi_\Ycal(g)=j_{0\ast}[\tilde q^{-1}(s_0)]-j_{\infty\ast}[\tilde q^{-1}(s_\infty)]$.  Since $\tilde q$ is flat, we can apply Proposition \ref{prop:pullback} to obtain $\divi_\Ycal(g)=\tilde q^\ast([s_0]-[s_\infty])$. Then $\tilde q_\ast \divi_\Ycal(g)=d \cdot ([s_0]-[s_\infty])$ for some $d\in\Z^2$ by Lemma \ref{lem:degree}. Now, $q_\ast \divi_\Ycal(g)=d\cdot i_\ast([s_0]-[s_\infty])$ and we finish because $\pi_\ast i_\ast[s_0]=i_\ast \pi_{\bos\ast}[s_0]=i_\ast \pi_{\bos\ast}[s_\infty]=\pi_\ast i_\ast [s_\infty]$.
\end{proof}

\begin{lemma}\label{lem:pushdiv} Let $f\colon\Xcal \to \Ycal$ be a proper and surjective morphism of supervarieties of the same even dimension.  
\marginnote{Corretto l'errore del morfismo finito (enunciato e dimostrazione)}
Assume moreover that $\Xcal$ has pure odd dimension 1 and that $\Oc_\Ycal \to f_\ast\Oc_\Xcal$ is injective. Then:
\begin{enumerate}
\item $\K(\Xcal)$ is a free $\K(\Ycal)$-module.
\item For every rational even superfunction $g\in \K(\Xcal)^\ast$ one has 
$$
f_\ast(\divi(g))=\divi(\ber(g))\,,
$$
where $\ber(g)\in \K(\Ycal)$ is the berezinian of the multiplication by $g$ in the free $\K(\Ycal)$-module $\K(\Xcal)$ (see \ref{sss:lattices}).
\end{enumerate}
\end{lemma}
\begin{proof}
(1) By Zariski's Main Theorem (Proposition \ref{prop:zariski}) we can assume that $\Ycal=\SSpec \As$ is affine and that $f$ is finite, so that $\Xcal=\SSpec \Bs$ and $f$ is induced by a finite injective morphism $\As\hookrightarrow \Bs$. Since $\Bs$ has pure odd dimension 1, $\Bs=B\oplus J_\Bs$ with $(J_\Bs)_\xi\simeq \eta K(B)$ where where $\xi$ is the generic point of $Y$. Then $J_\As^2=0$ so that $\As=A\oplus J_\As$. If $(J_\As)_\xi=0$, then $\K(\As)=K(A)$ which is a field and the result is automatic. Otherwise, $\K(\As)=K(A)\oplus(J_\As)_\xi\simeq K(A)\oplus\theta K(A)$ and the  image of $\theta$ in $\K(\Bs)$ generates $\eta\cdot K(B)$. It follows that if $g_1,\dots, g_n$ is a basis of $K(B)$ as a $K(A)$-vector space, it is also a free (even) basis of $\K(\Bs)$ as a $\K(\As)$-module. 

(3) If $Z$ is a subvariety of $Y$ of codimension $1$, one has to prove that the coefficients of $Z$ in $f_\ast(\divi (g))$ and in $\divi (\ber(g))$ are the same. If $\As_\xi$ and $\Bs_\xi$ are the localizations of $\As$ and $\bar\Bs$ at the generic point $\xi$ of $Z$, both are superdomains of even dimension 1 and the points $x_1,\dots,x_s$ of the fibre of $\xi$ are the generic points of the irreducible components of $\bar X$ that maps onto $Z$. Then, the coefficient of $Z$ in $\bar f_\ast(\divi (g))$ is $\sum_{i=1}^s \dim_{\kappa(\xi)}\kappa(x_i) \ord_{x_i}(g)$, which is equal to the coefficient of $[Z]$ in $\divi (\ber(g))$ by Proposition \ref{prop:preparation}.
\end{proof}

\begin{prop}\label{prop:push} If $f\colon\Xcal\to\Ycal$ is a proper surjective morphism of superschemes and $\alpha\in Z_h(\Xcal)$ is a supercycle rationally equivalent to 0, then $f_\ast (\alpha)$ is rationally equivalent to 0. Thus $f_\ast \colon Z_h(\Xcal) \to Z_h(\Ycal)$ induces a morphism of $\Z_2$-graded $\Z^2$-modules
$$
f_\ast \colon A_h(\Xcal) \to A_h(\Ycal)
$$
between the $\Z^2$-modules of $h$-supercycles modulo rational equivalence.
\end{prop}
\begin{proof} We can assume that $\alpha=\delta_\ast\divi (g)$ where $g$ is an even rational superfunction on a  subvariety $\delta\colon\Wc\hookrightarrow\Xcal$ of  even dimension $h+1$  and pure odd dimension $s=0$  or $s=1$  of $\Xcal$. When $s=0$, the situation is purely bosonic and the statement reduces to the ordinary (non super) one (\cite[Lemma 42.20.3]{Stacks} or \cite[Prop.\,1.4]{Fu98}). Suppose then that $\dim\Wc=h+1|1$.

If $\Wc'$ is the superscheme-theoretic image of the composition $\Wc\hookrightarrow\Xcal\xrightarrow{f}\Ycal$,  the restriction $f'\colon \Wc \to \Wc'$ is proper and surjective and the natural morphism $\Oc_{\Wc'}\to f'_\ast \Oc_\Wc$ is injective. It is enough to prove that either $f'_\ast \alpha=0$ or  $f'_\ast(\alpha)=[\divi(g')]$ for an even rational superfunction $g'\in \K(\Wc')$.

Let $h'$ be the dimension of $W'$. If $h'< h$, then automatically $f'_\ast \alpha=0$. 

We consider now the case $h'=h$. If $\xi$ is the generic point of $W'$, one has $\K(\Wc')=\Oc_{\Wc',\xi}$. The bosonic reduction of the superscheme $\SSpec \Oc_{\Wc',\xi}$ is $\{\xi\}=\Spec K(W')$. Consider the cartesian diagram
$$
\xymatrix{\Wc_{(\xi)} \ar[r] \ar[d]^g & \Wc \ar[d]^{f'}
\\
(\xi):=\SSpec \K(\Wc')\ar[r] & \Wc'
}
$$
The coefficient of $[W']$ in $f'_\ast(\divi g)$ is equal to the coefficient of $[\{\xi\}]$ in  $g_\ast(\divi g)$ because the superfield of rational superfunctions of $\SSpec  \K(\Wc')$ is $\K(\Wc')$. By Lemma \ref{lem:funcdiagram} $g_\ast(\divi g)=0$, which finishes this case.

The only remaining case is $\dim W'=h+1$. In this case we can apply Lemma \ref{lem:pushdiv} to complete the proof.
\end{proof}

\section{Stable supercurves}
\label{supercurves}
\subsection{Supercurves}

We adopt the following definition, that generalizes the definition of smooth superschemes of dimension $1|1$:
\begin{defin}
A supercurve is a reduced superscheme $\Xcal$ of pure even dimension 1 and pure odd dimension 1, according to Definition \ref{def:dim}.
\end{defin}
One then has \marginnote{$\Pi$? (in tutto il file)}
\begin{equation}\label{eq:scurveproj}
\Oc_\Xcal =\Oc_X\oplus \Lcl
\end{equation}
so that $\Xcal$ is projected. Since we are not assuming that $\Lcl$ is a line bundle, $\Xcal$ may fail to be split.

\begin{defin}\label{def:supercurve}
A relative supercurve is a flat morphisms of superschemes $f\colon \Xcal\to\Sc$ whose fibres are supercurves. 
\end{defin}

For ordinary curves (schemes of dimension 1), the CM condition is quite simple: a curve is CM if and only if it has neither isolated nor embedded points. One then has:

\begin{prop} For any   supercurve $f\colon \Xcal \to \Sc$, the bosonic curve $f_{bos}\colon X \to S$ is CM.
\end{prop}

\subsection{CM supercurves over ordinary schemes}

Any supercurve $f\colon \Xcal\to S$ over an ordinary scheme $S$ is projected: one has a morphism $\rho\colon \Xcal\to X$ of $S$-superschemes such that $\rho\circ j=\Id_X$. This is due to the fact that the decomposition given by Equation \eqref{eq:scurveproj} is still true in that case.

The dualizing complex of  a proper supercurve  $f\colon \Xcal \to S$ is then computed  using duality for the projection $\rho\colon\Xcal \to X$. Since the higher derived images of $\rho$ are zero, one has
\begin{align*}
	f^!(\Oc_S)\simeq p^!(\omega_{f_{bos}}[1])&\simeq \bR\Homsh_{\Oc_X}(\Oc_\Xcal, \omega_{f_{bos}}[1])\\
						 &\simeq (\omega_{f_{bos}}\oplus \bR\Homsh_{\Oc_X}(\Lcl, \omega_{f_{bos}}))[1]\,.
\end{align*}
One then has:
\begin{prop}\label{prop:CM}
 $f\colon \Xcal\to S$ is CM if and only if $\Extsh_{\Oc_X}^i(\Lcl,\omega_{f_{bos}})=0$ for $i\geq 1$.  In this case, the relative dualizing sheaf is 
$$
\omega_{f}\iso\omega_{f_{bos}}\oplus \Homsh_{\Oc_X}(\Lcl, \omega_{f_{bos}})\,,
$$
and the structure of module over $\Oc_\Xcal =\Oc_X\oplus \Lcl$ is given as follows: $\Oc_X$ acts naturally on the two factors; $\Lcl$ acts on $\omega_{f_{bos}}$ by zero and on $\Homsh_{\Oc_X}(\Lcl, \omega_{f_{bos}})$ as the evaluation
$$
\Lcl \otimes \Homsh_{\Oc_X}(\Lcl, \omega_{f_{bos}})\to \omega_{f_{bos}}\,.
$$
\qed
\end{prop}

Thus, we need to impose some conditions on the sheaf $\Lcl$  for $f\colon \Xcal\to S$ to be CM. 

\begin{prop} If $\Extsh_{\Oc_{X_s}}^i(\Lcl_{X_s},\omega_{X_s})=0$ for $i\geq 1$, the sheaves $\Lcl_{X_s}$ are of pure dimension 1 (and then are CM).
\end{prop}
\begin{proof} 
Since $X_s$ is CM, the dualizing sheaf $\omega_{X_s}$ is CM as well. 
Now, if $\Lcl_{\Xcal_s}$ is not pure, there exists a subsheaf $\M$ supported in dimension zero. Taking $\Homsh_{\Oc_{X_s}}(\ ,\omega_{X_c})$ in
$$
0\to \M \to \Lcl_{X_s} \to \Nc \to 0\,,
$$
we obtain 
$$
\Extsh_{\Oc_{X_s}}^1(\M,\omega_{X_s})\iso\Extsh_{\Oc_{X_s}}^2(\Nc,\omega_{X_s})
$$
and $\Extsh_{\Oc_{X_s}}^i(\Nc,\omega_{X_s})=0$ for $i\neq 0,2$.
Then, the local-to-global spectral sequence $E_2^{p.q}=H^p(X_s,\Extsh_{\Oc_{X_s}}^q(\Nc,\omega_{X_s}) \implies E_\infty^{p+q}=\Ext_{X_s}(\Nc,\omega_{X_s})$ has spherical fibres, and the corresponding exact sequence gives
$$
0= E_2^{0,2}=H^0(X_s,\Extsh_{\Oc_{X_s}}^2(\Nc,\omega_{X_s}))=H^0(X_s,\Extsh_{\Oc_{X_s}}^1(\M,\omega_{X_s}))\,.
$$
Since $\M$ is supported in dimension zero, this implies that $\Extsh_{\Oc_{X_s}}^1(\M,\omega_{X_s})=0$, which contradicts the fact that $\omega_{X_s}$ is a CM sheaf.
\end{proof}
\begin{corol} If $f\colon \Xcal\to S$ is CM, the even and odd components of the dualizing sheaf $\omega_{f}\iso\omega_{f_{bos}}\oplus \Homsh_{\Oc_X}(\Lcl, \omega_{f_{bos}})$ are relatively CM (relatively of pure dimension 1) sheaves on $X$.
\qed
\end{corol}
\begin{remark}
If $f\colon \Xcal \to S$ is smooth, then  $\omega_{f_{bos}}\iso\kappa_{X/S}$ and $\Lcl$ is locally free of rank one; moreover,
$$
\omega_{f}\iso \kappa_{X/S}\oplus (\Lcl^{-1}\otimes_{\Oc_X} \kappa_{X/S}) \iso  (\Lcl^{-1}\otimes_{\Oc_X} \kappa_{X/S})\otimes_{\Oc_X}\Oc_{\Xcal}^\Pi\,.
$$
The latter is a purely odd line bundle on $\Xcal$, and one recovers  the isomorphism
$\omega_{f}\iso \Ber_{f}$.
\end{remark}

\subsection{Singular SUSY curves with Ramond-Ramond punctures}

If $f\colon\Xcal \to \Sc$ is a proper smooth supercurve and $\Zc\hookrightarrow\Xcal $ is a positive relative  superdivisor \cite[Def. 3.2]{BrHR21}, the existence  of a superconformal structure $\Dcal\hookrightarrow \Theta_f=\Omega_f^\ast$ with RR punctures at $\Zc$ is   equivalent to the existence of a sheaf epimorphism 
$$
\Omega_f\xrightarrow{\bar\delta} \Ber_f(\Zc)\to 0
$$
such that the composition
$$
\ker\bar\delta \hookrightarrow \Omega_f\xrightarrow{d} \Omega_f\wedge \Omega_f\xrightarrow{\bar\delta\wedge\bar\delta} \Ber_f^{\otimes 2}(2\Zc)
$$
yields an isomorphism
$$
\ker\bar\delta \iso \Ber_f^{\otimes 2}(\Zc)\,,
$$
(see \cite[Prop. 3.5]{BrHR21}). The RR-superconformal  distribution $\Dcal$ is recovered as the image of the dual morphism $\bar\delta^\ast\colon \Ber_f^\ast(-\Zc)\hookrightarrow \Theta_f$. 

Since the relative Berezinian is the dualizing sheaf for $f$, we can generalize the notion of SUSY-curve with RR punctures to the singular case.

\begin{defin}\label{def:SUSYCM} Let $f\colon\Xcal \to \Sc$ be a proper CM relative curve and $\Zc\hookrightarrow\Xcal $  a relative positive superdivisor. We say that $f\colon\Xcal \to \Sc$ is a RR-SUSY curve along $\Zc$ if there exists an epimorphism
$$
\Omega_f\xrightarrow{\bar\delta} \omega_f(\Zc)\to 0\,,
$$
where $\omega_f$ is the relative dualizing sheaf, such that the composition
$$
\ker\bar\delta \hookrightarrow \Omega_f\xrightarrow{d} \Omega_f\wedge \Omega_f\xrightarrow{\bar\delta\wedge\bar\delta} \omega_f^{\otimes 2}(2\Zc)
$$
yields an isomorphism
$$
\ker\bar\delta \iso \omega_f^{\otimes 2}(\Zc)\,.
$$
\end{defin}

The irreducible components $\Zc_i$ of the superdivisor $\Zc$ are the \emph{Ramond-Ramond} punctures. 

\begin{remark} In what follows we shall assume that $\Zc$ is contained in the smooth locus $\U\to\Sc$ of $f$ and that it is not ramified over the base. Then, the $\Zc_i$'s have relative degree 1, do not intersect each other and  $\Zc$ is the disjoint union of them; we write $\Zc=\sum_i\Zc_i$.
\end{remark}

Note that in the CM case $\bar\delta$ cannot  in general be recovered   from the superconformal distribution $\Dcal$ as $\Theta_f^\ast$ may fail to be equal to $\Omega_f$.

Let $f\colon\Xcal \to S$ be an RR-SUSY curve along $\Zc$ over an ordinary scheme $S$, so that $\Xcal$ is projected and $\Oc_\Xcal\simeq \Oc_X\oplus \Lcl$. Taking into account the expression for $\omega_f$ given by Proposition \ref{prop:CM}, the surjective derivation $\delta\colon\Oc_\Xcal \to \omega_f(\Zc)$ is equivalent to the following data:
\begin{enumerate}
\item a surjective derivation $\delta_+\colon \Oc_X\to \omega_{f_{bos}}(Z)$ over $\Oc_S$;
\item an epimorphism of $\Oc_X$-modules $\delta_-\colon\Lcl\to \Homsh_{\Oc_X}(\Lcl, \omega_{f_{bos}}(Z))$.
\end{enumerate}

\begin{remark} If the bosonic fibres are Gorenstein (for instance, if they are nodal curves), the dualizing sheaf $\omega_{f_{bos}}$ is a line bundle.
 If in addition $\Lcl$ is a line bundle on $X$, then
\begin{enumerate}
\item $\Xcal$ is split;
\item $\omega_f$ is a line bundle on $\Xcal$ of rank $0|1$ (Proposition 5.4);
\item $\Dcal$ is a locally free subsheaf of rank $0|1$ of $\Theta_f$.
\end{enumerate}
In this case, \emph{the distribution $\Dcal$ determines $\bar\delta$}, because $\bar\delta$ is the composition of the natural morphism $\Omega_f\to \Omega_f^{\ast\ast}$ with the dual morphism
$\Omega_f^{\ast\ast}=\Theta_f^\ast \to \Dcal^\ast=(\omega_f(\Zc))^{\ast\ast}=\omega_f(\Zc)$ of the immersion $\Dcal\hookrightarrow \Theta_f$.
\label{rem:rifatto}
\end{remark}

\subsection{Pre-stable and stable supercurves}

Fix an algebraically closed field $k$.

Recall that a (ordinary) \emph{pre-stable} curve over $k$ is a proper connected curve over $k$ whose singularities are simple nodes \cite[Chap.\,III, Def.\,2.1]{Ma99}. A  pre-stable $n$-pointed curve is a pair $(X,D)$, where $X$ is a pre-stable curve and $D$ is an ordered family of $n$ different points $x_1,\dots,x_n$. \marginnote{Rimessa definizione di curva prestabile}

 As we have already mentioned, every pre-stable curve $X$ is Gorenstein. The reason for that is that pre-stable curves are locally complete intersections. Then the dualizing sheaf $\omega_X$ is a line bundle.

The notion of stability of $n$-pointed curves can be given in several equivalent ways. One of them is to say that a pre-stable $n$-pointed curve is \emph{stable} when the line bundle $\omega_X(D)$ is ample. This is equivalent to any of the following conditions \cite{DelMum69}:
\begin{enumerate}
\item
A pre-stable $n$-pointed curve $(X,D)$ is stable if and only if any rational component of the normalization $\tilde X$ contains at least three points from $\pi^{-1}(D\cup X_{sing})$, (where $\pi\colon\tilde X\to X$ is the normalization morphism) and any component $\tilde X_i$ with genus $g(\tilde X_i) = 1$ contains at least one such point.
\item A pre-stable $n$-pointed curve $(X,D)$ is stable if and only if for any component $\tilde X_i$ of $\tilde X$ one has $2 g(\tilde X_i)-2 + n_i>0$ where $n_i$ is the number of marked points (the ones in $\pi^{-1}(D\cup X_{sing})$) contained in $\tilde X_i$.
\item A pre-stable $n$-pointed curve $(X,D)$ is stable if and only if it has a finite number of automorphisms.
\end{enumerate}

We now recall, with a slight reformulation, the definition of stable supercurves with Neveu-Schwarz (NS) and Ramond-Ramond (RR) punctures \cite{Del87,FKP20}.

\begin{defin}\label{def:stablecurve}
 A stable (resp. pre-stable) SUSY curve of arithmetic genus $g$ with punctures over a superscheme $\Sc$ is a proper and Cohen-Macaulay (Definition \ref{def:CM}) $\Sc$-supercurve $f\colon \Xcal \to \Sc$  (Definition \ref{def:supercurve}), together with:
\begin{enumerate}
\item A collection of disjoint closed sub-superschemes $\Xcal_i\hookrightarrow \U$ ($i=1,\dots,\nf_{NS}$), where $\U$ is the smooth locus of $f$, such that $\pi\colon \Xcal_i\to\Sc$ is an isomorphism for every $i$. They are called   \emph{NS-punctures}.
\item A collection of disjoint Cartier divisors $\Zc_j$ of relative degree 1 ($j=1,\dots,\nf_R$), contained in $\U$. They are called   \emph{RR-punctures}. 
\item An epimorphism $\bar\delta\colon \Omega_f\to \omega_f(\Zc)$, where $\Zc=\sum \Zc_j$, satisfying the conditions of Definition \ref{def:SUSYCM}.
\end{enumerate}
Moreover, these data have to fulfil the following condition: if for every bosonic fibre $X_s$ of $f\colon \Xcal \to \Sc$ we write $x_{s,i}=X_i\cap X_s$ and $z_{s,j}=Z_j\cap X_s$,  then the pair $(X_s, D_s)$ with $D_s= \{x_{s,1}\dots x_{s,\nf_{NS}},z_{s,1}\dots, z_{s,\nf_{R}}\}$ is a stable (resp. pre-stable) $(\nf_{NS}+\nf_R)$-pointed curve of arithmetic genus $g$.
\end{defin}

{
\begin{defin}\label{def:prestmorphs} Let $(f\colon\Xcal\to\Sc, \{\Xcal_i\},\{\Zc_j\}, \bar\delta)$ and $(f\colon\Xcal'\to\Sc, \{\Xcal'_i\},\{\Zc'_j\}, \bar\delta)$ be pre-stable SUSY curves over $\Sc$. A morphism between them is a morphism of RR-SUSY curves $f\colon \Xcal\to\Xcal'$ over $\Sc$ (\cite[Def.\,3.9]{BrHR21}) preserving the NS-punctures, that is, $f(\{\Xcal_i\})\subseteq\{\Xcal'_i\}$.
\end{defin}
\begin{remark} According to \cite[Def.\,3.9]{BrHR21}, the number of RR-punctures of the two curves  must be the same. Regarding the NS-punctures it may be sensible to impose that $f$ maps $\{\Xcal_i\}$ surjectively onto $\{\Xcal'_i\}$; in such a case, one has $\nf_{NS}\geq\nf'_{NS}$.
\end{remark}
}
The moduli of complex stable supercurves has the structure of a Deligne-Mumford superstack, that is, it is a DM-stack on the category of superschemes (see \cite{CodViv17} for the  theory of superstacks). The precise result is:

\begin{thm}[\cite{FKP20}]\label{thm:moduli}
There is a proper and smooth DM-superstack $\bar\Mf_{g,\nf_{NS},\nf_R}$ over $\C$ representing the functor of stable supercurves of arithmetic genus $g$ with $\nf_{NS}$ NS-punctures and $\nf_R$ RR punctures.
\qed
\end{thm}

\section{Stable supermaps}\label{s:stsupermaps}
\subsection{Definitions}

We would like to have a definition of stable supermaps with values in a superscheme $\Ycal$, which, when $\Ycal$ is an ordinary single point, coincides with the notion of a stable SUSY-curve with punctures. We fix an algebraically closed field $k$ and only consider superschemes over $k$.

Let $\Ycal$ be a proper smooth superscheme and $\beta\in A_1(\Ycal)$ a rational equivalence class of  supercycles.
{
\begin{defin}\label{def:stablemap} A stable supermap  into $\Ycal$ of class $\beta$  over a superscheme $\Sc$, with $\nf_{NS}$ NS-punctures and $\nf_R$ RR-punctures is given by the following data:
\begin{enumerate}
\item A pre-stable SUSY-curve $(f\colon\Xcal\to\Sc, \{\Xcal_i\},\{\Zc_j\}, \bar\delta)$ over $\Sc$ with $\nf_{NS}$ NS-punctures and $\nf_R$ RR-punctures (Definition \ref{def:stablecurve}).
\item A superscheme morphism $\phi\colon\Xcal\to \Ycal$ such that $\phi_\ast[\Xcal_s]=\beta$ for every closed point $s\in S$ (cf.~Proposition \ref{prop:push}).
\item For every geometric point $s\in S$, if $\tilde X'_s$ is a component of the normalization $\pi\colon \tilde X_s\to X_s$ of the bosonic fibre $X_s$ which is contracted by $\phi\circ\pi$ to a single point, then 
\begin{enumerate}
\item if $\tilde X'_s$ is rational,   it contains at least three points from $\pi^{-1}(D_s\cup X'_{s,sing})$, where $D_s= \{x_{s,1}\dots x_{s,\nf_{NS}},z_{s,1}\dots, z_{s,\nf_{R}}\}$ with $x_{s,i}=X_i\cap X_s$ and $z_{s,j}=Z_j\cap X_s$.
\item if $\tilde X'_s$ has genus 1,   it contains at least one such point.
\end{enumerate}
\end{enumerate}
\end{defin}
}

\begin{remark} The supercycle class $\phi_\ast[\Xcal_s]$ is defined even if $\phi$ is not proper, because the restriction of $\phi$ to $\Xcal_s$ is automatically proper.
\end{remark}
\begin{remark}
When $\Ycal$ is a single point, the unique morphism $\phi$ to $\Ycal$ is the natural projection and  $\beta$ has to be zero, so that  the second condition is automatically fulfilled and stable supermaps into a point are the same thing as stable SUSY-curves. 
\end{remark}
\begin{remark} This definition is different from the one given in \cite[Def. 3.1]{AdGr20} because  there the fibres of $\Xcal\to\Sc$ are assumed to be bosonic. With that definition   stable supermaps over a single point are merely the ordinary stable curves and the superstructure is lost.
\end{remark}

Our aim is to prove an analogue of  Theorem \ref{thm:moduli} for stable supermaps.

Let $\Ycal$ be a superscheme; we fix non-negative integers $g$, $\nf_{NS}$ and $\nf_R$ and a cycle $\beta\in A_1(\Ycal)$.  
Denote by $\Sf$ the category of superschemes. 
\begin{defin}\label{def:CFGsmaps} The category fibred in groupoids (CFG) 
$$
\Sf\Mf_{g,\nf_{NS},\nf_R}(\Ycal,\beta) \xrightarrow{p} \Sf
$$
of $\beta$-valued stable supermaps of arithmetic genus $g$ with $\nf_{NS}$ NS-punctures and $\nf_R$ RR-punctures into $\Ycal$,
is given by the following data: \
\begin{itemize}\item Objects are $\beta$-valued stable supermaps $\Xf=((f\colon\Xcal\to\Sc, \{\Xcal_i\},\{\Zc_j\}, \bar\delta), \phi)$ \footnote{We use a simplified notation here, some of the relevant  data are left implicit.} over a superscheme $\Sc$ into $\Ycal$ of arithmetic genus $g$, with $\nf_{NS}$ NS-punctures, $\nf_R$ RR-punctures (Definition~\ref{def:stablemap}). 
\item Morphisms are cartesian diagrams
\begin{equation}\label{eq:cartesian}
\xymatrix{
\Xf'\ar[r]^\Xi\ar[d]_{f'} & \Xf \ar[d]^f\\
\Tc\ar[r]^\xi & \Sc
}
\end{equation}
that is, diagrams inducing an isomorphism $\Xcal'\iso \xi^\ast \Xcal:=\Xcal\times_{\Sc}\Tc$ of superschemes over $\Tc$ compatible with the punctures and the morphisms $\bar\delta$ and such that $\phi'=\phi\circ\Xi$. 
The functor $p$  maps a stable supermap to the base superscheme  and the morphism $\Xi$ to the base morphism $\xi$. The pullback $\xi^\ast\Xf\to\Tc$ is given by the fibre product $f_\Tc\colon\Xcal\times_{\Sc}\Tc\to\Tc$ and the fibre products of the data $\{\Xcal_i\},\{\Zc_j\}, \bar\delta$\footnote{The fibre product of $\delta$ is a derivation of the same kind due to Proposition \ref{prop:bcCM}.}(thus providing  a natural ``cleavage''). \end{itemize}
\end{defin}

As it is customary, for every superscheme $\Sc$ we denote by $\Sf\Mf_{g,\nf_{NS},\nf_R}(\Ycal,\beta)(\Sc)$, and call it the \emph{fibre of $\Sf\Mf_{g,\nf_{NS},\nf_R}(\Ycal,\beta)$ over $\Sc$},  the category whose objects are stable supermaps over $\Sc$ and whose morphisms $\Xi$ lie over the identity $\xi=\Id$.

Since \'etale descent data for stable supermaps are effective and descent data for morphims to $\Ycal$ are effective as well, we easily see that descent data for $\Sf\Mf_{g,\nf_{NS},\nf_R}(\Ycal,\beta)$ are effective. One easily checks that the isomorphisms between two objects of $\Sf\Mf_{g,\nf_{NS},\nf_R}(\Ycal,\beta)$ form a sheaf in the \'etale topology of superschemes, so that the CFG $\Sf\Mf_{g,\nf_{NS},\nf_R}(\Ycal,\beta)$ is a \emph{superstack}.

\section{SUSY Nori motives}\label{s:Norimotives}

\subsection{Nori geometric categories}\label{ss:Norigeomcats}

The basic combinatorial objects in the theory of Nori motives are categories of diagrams \cite{MaMar20}. In \cite{BoMa08} and \cite{Ma99} a slightly changed formalism refers to graphs, and we will use the formalism of \cite{BoMa08, Ma99} rather than \cite{MaMar20}.

An object of such a category, a graph $\tau$, is a family of structures  $(F_\tau,V_\tau,\partial_\tau,j_\tau)$, just sets or structured sets and maps between them in the simplest cases. Elements of $F_\tau$, resp.\,$V_\tau$, are called \emph{flags}, resp.\,\emph{vertices} of $\tau$. The map $\partial_\tau\colon F_\tau \to V_\tau$ associates to each flag a vertex, called its \emph{boundary}. The map $j_\tau\colon F_\tau \to F_\tau$ must satisfy the condition $j_\tau^2 = \Id$, the indentity map of $F_\tau$ to itself. Pairs of flags $(f,j_\tau f)$ are called \emph{edges}, connecting boundary vertices of these two flags.

A morphism of graphs $\sigma \to \tau$ consists of two maps $F_\sigma \to F_\tau$, $V_\sigma\to V_\tau$ compatible with the $\partial$'s and $j$'s. 

To each (small) category $\Cc$ one can associate its graph $D(\Cc)$, whose vertices are objects of $\Cc$ and flags are morphisms $\ast\to X$ and $X \to\ast$, where $\ast$ runs over all objects of $\Cc$. According to \cite[Def.\, 0.1.1]{MaMar20}, edges, oriented from $X$ to $Y$, are represented by diagrams $X \to Z \to Y$, that is, by decompositions of morphisms from $X$ to $Y$, presented as a product of two morphisms.

Functors between two categories produce morphisms between their diagrams.

\subsection{SUSY Nori geometric categories}\label{ss:Noricats}

As in \cite[0.1-0.6]{MaMar20},  we need to chose  a ``geometric'' category $\Cc$ of superschemes/superstacks/supermaps/..., additionally endowed with classes of morphisms of ``closed embeddings'' $Y\to X$ and complements to closed embeddings $X \setminus Y\to X$, satisfying certain restrictions.

For us, roughly speaking, such $\Cc$ is a category of stable supermaps (Definition \ref{def:stablemap}). 
If we go to a more detailed definition, at least  two possibilitites arise.

\subsubsection{First possibility}
The objects of this category are the same of the objects of the various categories $\Sf\Mf_{g,\nf_{NS},\nf_R}(\Ycal,\beta)$ of Definition \ref{def:CFGsmaps}, but we have to relax the definition of the relevant morphisms. For technical reasons we always assume that the superschemes $\Ycal$ are smooth and proper.

\begin{defin}\label{def:stablemapsmorphism} Let $\Ycal$, $\Ycal'$ be two proper and smooth superschemes and $\beta\in A_1(\Ycal)$, $\beta'\in A_1(\Ycal')$. Let 
$F :=((f\colon\Xcal\to\Sc, \{\Xcal_i\},\{\Zc_j\}, \bar\delta),\phi\colon \Xcal\to\Ycal, \beta)$ and 
$F':=((f'\colon\Xcal'\to\Sc, \{\Xcal'_i\},\{\Zc'_j\}, \bar\delta'),\phi'\colon \Xcal'\to\Ycal', \beta')$
be  stable supermaps over $\Sc$ (where the genus and the number of punctures of one and the other are allowed to be different). A morphism $\mathbf{\Phi}\colon F \to F'$ is a pair $\mathbf{\Phi}=(g, \psi)$ where 
\begin{enumerate}
\item $g\colon \Xcal \to \Xcal'$ is a morphism of punctured RR-SUSY curves over $\Sc$ \cite[Def.\,3.9 ]{BrHR21} preserving the NS punctures (that is, $g( \{\Xcal_i\} )\subseteq \{\Xcal'_i\})$.
\item $\psi\colon \Ycal \to \Ycal'$ is a morphism of superschemes such that $\psi\circ\phi= \phi'\circ g$ and $\psi_\ast\beta=\beta'$ 
(notice also that $\psi$ is automatically proper).
\end{enumerate}
Here we use a slight generalization of \cite[Def.\,3.9 ]{BrHR21} to allow the number of RR punctures of the two SUSY curves to be different, namely, we only impose that $g( \{\Zc_j\} )\subseteq \{\Zc'_j\})$.
\end{defin}

It is clear from the definition that the composition of two morphisms of stable supermaps, defined as $(\bar g,\bar\psi)\circ (g, \psi)=(\bar g \circ g, \bar\psi\circ\psi)$ is   morphism of  stable supermaps.

One then has a category for every superscheme $\Sc$, the category $\Sf\Mf(\Sc)$ of stable supermaps over $\Sc$. It is naturally fibered
$$
q\colon\Sf\Mf(\Sc) \to \Sf_1
$$
over the category $\Sf_1 $ whose objects are pairs $(\Ycal,\beta)$, where $\Ycal$ is a proper smooth superscheme and $\beta\in A_1(\Ycal)$, with morphisms $(\Ycal,\beta)\to (\Ycal',\beta')$ given by superscheme morphisms $\psi\colon \Ycal\to\Ycal'$ (automatically proper) such that $\psi_\ast\beta=\beta'$. 

For every superscheme morphism $\xi\colon \Tc \to \Sc$ there is a  pullback    functor
$$
\xi^\ast\colon \Sf\Mf(\Sc) \to \Sf\Mf(\Tc)
$$
compatible with the fibration functors to $\Sf_1$. One can see that  for every superscheme $\Sc$, the functor $q\colon\Sf\Mf(\Sc) \to \Sf_1$ is a CFG and that for every base-change $\xi\colon \Tc \to \Sc$ the functor $\xi^\ast$ is cocartesian (see \cite[Thm.\,3.6]{BeMa96}).

To conclude this part we give the definition of morphism of stable supermaps over different base superschemes. 
\begin{defin}\label{def:stmpmor}
If $F =((f\colon\Xcal\to\Sc, \{\Xcal_i\},\{\Zc_j\}, \bar\delta),\phi\colon \Xcal\to\Ycal,\beta)$ and 
$F'=((f'\colon\Xcal'\to\Sc', \{\Xcal'_i\},\{\Zc'_j\}, \bar\delta',\beta'),\phi'\colon \Xcal'\to\Ycal')$ are stable supermaps, a morphism $\mathbf{\Psi}\colon\F \to F'$ is a pair $(\xi,\mathbf{\Phi})$ where $\xi\colon \Sc \to \Sc'$ is a superscheme morphism and $\mathbf{\Phi}\colon F \to \xi^\ast F'$ is a morphism of stable supermaps over $\Sc$.
\end{defin}

Such a morphism is determined by a morphism $g\colon \Xcal \to \Xcal'$ covering $\xi$ ($f'\circ g=\xi\circ f$) that preserves the SUSY structure and the punctures, and a morphism $\psi\colon \Ycal\to \Ycal'$ such that $\psi\beta=\beta'$ and $\psi\circ\phi=\phi'$. As in the case of stable supermaps over the same base superscheme, the composition of two morphisms of stable supermaps is a morphism of stable supermaps; more precisely, if $F$, $F'$ and $F''$ are stable supermaps over superschemes $\Sc$, $\Sc'$ and $\Sc''$, respectively, and $(\xi,\mathbf{\Phi})\colon F \to F'$, $(\xi',\mathbf{\Phi'})\colon F'\to F''$ are morphisms, the composition of them $(\xi,\mathbf{\Phi})\circ (\xi',\mathbf{\Phi'})$ is the morphism of stable supermaps given by $(\xi\circ\xi',\mathbf{\Phi}\circ \mathbf{\Phi'})$.

\subsubsection{Second possibility}\label{sss:Noricats}
 
Now we start with a pre-stable supercurve $(f\colon\Xcal\to\Sc, \{\Xcal_i\},\{\Zc_j\}, \bar\delta)$ as in Definition \ref{def:stablecurve}, and a superscheme morphism $\phi\colon \Xcal \to \Ycal$, where $\Ycal$ is proper and smooth. We also fix a class $\beta\in A_1(\Ycal)$, but  without imposing conditions on the push-forward under $\phi$ of the fundamental cycles of the fibres.

\begin{defin}\label{def:betagood} We say that a subsuperscheme\footnote{Since we are considering only locally noetherian superschemes, such immersions are locally closed for the Zariski topology.}  $\Sc'\hookrightarrow \Sc$ is $\beta$-good  if the restriction of $(f\colon\Xcal\to\Sc, \{\Xcal_i\},\{\Zc_j\}, \bar\delta)$ and $\phi$ to $\Sc'$ is a stable supermap of class $\beta$ (Definition \ref{def:stablemap}). In particular, one has $\phi_\ast[\Xcal_s]=\beta$ for every geometric point $s'\in \Sc'$.
\end{defin}
We consider the  category $\Sf_\beta(\Sc)$ whose objects are the $\beta$-good subsuperschemes of $\Sc$ and whose morphisms are the natural embeddings (when they exist). $\Sf_\beta(\Sc)$ is a directed category to $\Sc$, with the property that given morphisms $\Sc_1\hookrightarrow \Sc_2\hookrightarrow\Sc_3$ and $\Sc_1\hookrightarrow \Sc'_2\hookrightarrow\Sc_3$,
the compositions coincide as both are the embedding of $\Sc_1$ into $\Sc_3$. Moreover, the disjoint  union of two  $\beta$-good subsuperschemes is also $\beta$-good. Thus $\Sf_\beta(\Sc)$ is a monoidal category.
 
\begin{remark}
This definition is very rigid. We may allow automorphisms of the class $\beta$, that is, automorphisms $\psi\colon\Ycal\to\Ycal$ such that $\psi_\ast\beta=\beta$, or automorphisms of the pre-stable curve $f_{\Sc'}\colon\Xcal_{\Sc'}\to\Sc'$ for any $\beta$-good subsuperscheme $\Sc'$ of $\Sc$.  In this case, the definition of morphisms in the category $\Sf_\beta(\Sc)$ is similar to that in Definition \ref{def:stablemapsmorphism}.
\end{remark}

\subsection{SUSY Nori closed embeddings}\label{sss:Noriclosed}

Again, there are  different possible definitions of Nori closed embeddings for stable supermaps. 

We write some of them according to the types of SUSY Nori geometric categories described in Section \ref{ss:Noricats}.

\subsubsection{First possibility (type 1)}

One possible choice is the following.
\begin{defin}\label{def:Noriclosed1} Let $F =((f\colon\Xcal\to\Sc, \{\Xcal_i\},\{\Zc_j\}, \bar\delta),\phi\colon \Xcal\to\Ycal, \beta)$ and 
$F'=((f'\colon\Xcal'\to\Sc', \{\Xcal'_i\},\{\Zc'_j\}, \bar\delta'),\phi'\colon \Xcal'\to\Ycal',\beta')$ be stable supermaps. A morphism $\mathbf{\Psi}=(\xi,\mathbf{\Phi})\colon F \to F'$ is a Nori closed embedding (of type 1) if the involved morphisms $\xi\colon \Sc\to\Sc'$, $g\colon \Xcal \to \Xcal'$ and $\psi\colon\Ycal\to\Ycal'$ are closed immersions of superschemes. We write $\mathbf{\Psi}\colon F \hookrightarrow F'$ or $ F \subset F'$ when the morphisms are understood.
\end{defin}

Note that if $\mathbf{\Psi}\colon F \hookrightarrow F'$ is a closed embedding of stable supermaps, the number of NS-punctures and  RR-punctures of $f\colon\Xcal \to\Sc$ must be smaller that the corresponding numbers of punctures of $f'\colon\Xcal' \to\Sc'$, that is, $\nf_{NS}\leq \nf'_{NS}$, $\nf_R\leq \nf'_R$. Moreover, for every geometric point $s\in \Sc$ the fibre $\Xcal_s$ is isomorphically mapped to a subcurve of $\Xcal'_{\xi(s)}$ consisting of some of its irreducible components.

\begin{example} If $F'=((f'\colon\Xcal'\to\Sc', \{\Xcal'_i\},\{\Zc'_j\}, \bar\delta',\beta'),\phi'\colon \Xcal'\to\Ycal')$ is a stable map over $\Sc'$ and and $\xi\colon \Sc\hookrightarrow \Sc'$ is a closed immersion, the natural base change morphism $\xi^\ast F' \to F'$ is a closed embedding.
\end{example} 

\begin{example} 
Given a stable map $F =((f\colon\Xcal\to\Sc, \{\Xcal_i\},\{\Zc_j\}, \bar\delta),\phi\colon \Xcal\to\Ycal,\beta)$, every morphism $\psi\colon \Ycal \to \Ycal'$ gives rise to a push-forward stable map $\psi_\ast F =((f\colon\Xcal\to\Sc, \{\Xcal_i\},\{\Zc_j\}, \bar\delta),\psi\circ \phi\colon \Xcal\to\Ycal', \psi_\ast\beta)$. When $\psi$ is a closed embedding,  the induced morphism $F \to \psi_\ast F$ is a closed embedding as well.
\end{example}

\begin{example} Let $F =((\Xcal, \{x_i\},\{\Zc_j\}, \bar\delta),\phi\colon \Xcal\to\Ycal,\beta)$ and $F =((\Xcal', \{x'_i\},\{\Zc'_j\}, \bar\delta'),\break \phi\colon \Xcal'\to\Ycal,\beta)$ be stable maps over a single (geometric) point and $F \hookrightarrow F'$ a closed embedding, where the corresponding morphism $\psi\colon \Ycal \to \Ycal$ is the identity. Then $g\colon \Xcal\hookrightarrow\Xcal'$ is a closed embedding which identifies $\Xcal$ with a union of some irreducible components of $\Xcal'$. The condition $\phi_\ast[\Xcal']=\beta = \phi_\ast[\Xcal]$ implies that $\phi$ contracts the superschematic closure of $\Xcal'\setminus \Xcal$ to points.
\end{example}

{\subsubsection{Type 2}
In this case we fix a superscheme $\Sc$,  a pre-stable supercurve $f\colon \Xcal \to \Sc$  and a superscheme morphism $\phi\colon \Xcal \to \Ycal$, where $\Ycal$ is proper and smooth. We also fix a class $\beta\in A_1(\Ycal)$. Our category is now is the category $\Sf_\beta(\Sc)$ of Subsection \ref{sss:Noricats}.}
{
\begin{defin}\label{def:Noriclosed2} A Nori closed embedding (of type 2) is simply a closed embedding of $\beta$-good subsuperschemes of $\Sc$ (Definition \ref{def:betagood}).
\end{defin}
}

\subsection{SUSY Nori motives} The last intermediate step in the construction of Nori motives of a geometric category $\Cc$ consists in the introduction of the graphs of effective pairs \cite[Def.\,9.1.1]{HuM-St17}.

We describe SUSY Nori diagram of effective pairs \cite[Ch.\,9]{HuM-St17}, \cite{MaMar20} for the two types of Nori closed embeddings of Subsection \ref{sss:Noriclosed}.

\subsubsection{Type 1}\label{sss:type1} In this case, a SUSY Nori diagram has the following structure.

\begin{enumerate}
\item
 Vertices are given by pairs $(\mathbf{\Psi},i)$, where $\mathbf{\Psi}\colon F\to F'$
is a morphism of stable supermaps which is a closed embedding according to Definition \ref{def:Noriclosed1},
and $i$ is an integer number. We also write $(F',F,\mathbf{\Psi},i)$ for $(\mathbf{\Psi},i)$.
\item
Given vertices  $(F',F,\mathbf{\Psi},i)$ and $(\bar F',\bar F,\mathbf{\bar\Psi},i)$, a pair of closed embeddings $\mathbf{\Lambda}\colon F \hookrightarrow \bar F$, $\mathbf{\Lambda'}\colon F' \hookrightarrow \bar F'$ (in the sense of Definition \ref{def:Noriclosed1}) such that the diagram
$$\xymatrix{ 
F \ar@{^{(}->}[r]^{\mathbf{\Psi}}\ar@{^{(}->}[d]^{\mathbf{\Lambda}} & F' \ar@{^{(}->}[d] ^{\mathbf{\Lambda'}} \\
\bar F \ar@{^{(}->}[r]^{\mathbf{\bar\Psi}} & \bar F'
}
$$
commutes, produces an edge
$$
(\mathbf{\Lambda'},\mathbf{\Lambda})^\ast\colon (\bar F',\bar F,\mathbf{\bar\Psi},i)\to (F',F,\mathbf{\Psi},i)\,.
$$
\item
For every chain $F\xrightarrow{\mathbf{\Psi}} F' \xrightarrow{\mathbf{\Psi'}} F''$ of Nori closed embeddings there is an edge
$$
\partial\colon (F',F,\mathbf{\Psi},i)\to (F'',F',\mathbf{\Psi'},i+1)\,.
$$
\end{enumerate}

\subsubsection{Type 2} Let $(f\colon\Xcal\to\Sc, \{\Xcal_i\},\{\Zc_j\}, \bar\delta)$ be a pre-stable supercurve, $\phi\colon \Xcal \to \Ycal$ a superscheme morphism, where $\Ycal$ is proper and smooth, and $\beta\in A_1(\Ycal)$. We shorten these data as $F_{\Sc,f,\phi,\beta}$.
\begin{defin}\label{def:Noridiag2} The SUSY Nori diagram $D_{\mbox{\rm \tiny eff}}(F_{\Sc,f,\phi,\beta})$ of effective pairs for $F_{\Sc,f,\phi,\beta}$ is given by:
\begin{enumerate}
\item A vertex is a triple $(\Sc_1,\Sc_2,i)$ where $\Sc_2\hookrightarrow\Sc_1$ is closed embedding of $\beta$-good subsuperschemes of $\Sc$ and $i$ is a non-negative integer number.
\item There are two types of edges (besides the obvious identities):
\begin{enumerate}
\item Given two vertices $(\Sc_1,\Sc_2,i)$, $(\Sc'_1,\Sc'_2,i)$, an immersion $j\colon\Sc'_1\hookrightarrow \Sc_1$ of subsuperschemes of $\Sc$ that induces an immersion $j_{|\Sc'_2}\colon \Sc'_2 \hookrightarrow \Sc_2$, produces an edge
$$
h^\ast\colon(\Sc_1,\Sc_2,i)\to (\Sc'_1,\Sc'_2,i)\,.
$$
\item For every chain $\Sc_3\hookrightarrow\Sc_2\hookrightarrow\Sc_1$ of $\beta$-good subsuperschemes of $\Sc$, there is an edge
$$
\partial\colon (\Sc_2,\Sc_3,i) \to (\Sc_1,\Sc_2,i+1)\,.
$$
\end{enumerate}
\end{enumerate}
\end{defin}
 For the next step we need to have a representation $T$ of $D_{\mbox{\rm \tiny eff}}(F_{\Sc,f,\phi,\beta})$ into an appropriate abelian category in the sense of  \cite[Def. 7.1.4]{HuM-St17}. This could be the category of modules over a commutative ring $A$ or over a superring $\As$ \cite[2.1]{BrHR21}. Assuming that   such a representation is given, one can define:
\begin{defin}\label{def:effNori} The category of effective mixed SUSY Nori motives of $F_{\Sc,f,\phi,\beta}$ with respect to the representation $T$ is the diagram category $\M\M_{SUSY-Nori}^{\mbox{\rm \tiny eff}}=\Cc(D_{\mbox{\rm \tiny eff}}(F_{\Sc,f,\phi,\beta}),T)$ (see \cite[Def.\,7.1.10]{HuM-St17}).
\end{defin}


\subsubsection{Effective mixed SUSY Nori motives}\label{ss:effNori}

We now    build some   representations of SUSY Nori diagrams in an abelian category in the sense of \cite[Def. 7.1.4]{HuM-St17},  thus constructing concrete examples of effective SUSY mixed Nori motives.

\begin{example} [Topological effective Type 1  mixed SUSY Nori motives]\label{ex:effNori1}
 Let $F =((f:\Xcal\to\Sc, \{\Xcal_i\},\{\Zc_j\}, \bar\delta),\phi: \Xcal\to\Ycal, \beta)$ and 
$F'=((f':\Xcal'\to\Sc', \{\Xcal'_i\}, \{\Zc'_j\}, \bar\delta'),\phi': \Xcal'\to\Ycal',\beta')$ be stable supermaps, and let $\mathbf{\Psi}\colon F\to F'$ a closed embedding (Definition \ref{def:Noriclosed1}). Recall that $\mathbf{\Psi}$ is determined by closed immersions of superschemes $\xi\colon \Sc\hookrightarrow \Sc'$, $g\colon \Xcal \hookrightarrow  \Xcal'$ and $\psi\colon\Ycal\hookrightarrow \Ycal'$ such that $f'\circ g=\xi\circ f$, $\psi_\ast\beta=\beta'$, $\psi\circ\phi=\phi'$ and $g$ preserves the SUSY structure and the punctures. 

 We have an open immersion $\iota\colon V=X'/g(X)\hookrightarrow X'$. We define a representation $T$ in the category of abelian groups by setting
$$
T(F',F,\mathbf{\Psi},i)=\Hs^i(\bR \phi'_\ast(\iota_! \Z_{V})) 
$$
for the vertex $(F',F,\mathbf{\Psi},i)$.

Now, consider   vertices $(F',F,\mathbf{\Psi},i)$, $(\bar F',\bar F,\mathbf{\bar\Psi},i)$ and a pair of closed embeddings $\mathbf{\Lambda}\colon F \hookrightarrow \bar F$, $\mathbf{\Lambda'}\colon F' \hookrightarrow \bar F'$ such that the diagram
$$\xymatrix{ 
F \ar@{^{(}->}[r]^{\mathbf{\Psi}}\ar@{^{(}->}[d]^{\mathbf{\Lambda}} & F' \ar@{^{(}->}[d] ^{\mathbf{\Lambda'}} \\
\bar F \ar@{^{(}->}[r]^{\mathbf{\bar\Psi}} & \bar F'
}
$$
commutes, as in (3) of  \ref{sss:type1}.
As we have said, $\mathbf{\Psi}$ is determined by closed immersions $\xi$, $g$ and $\psi$; we denote respectively with bars, primes, and bars together with primes,  the corresponding associated morphisms to $\mathbf{\Lambda}$, $\mathbf{\Lambda'}$ and $\mathbf{\bar\Psi}$, respectively. Then, if $\bar\iota\colon \bar V=\bar X'/\bar g'(\bar X)\hookrightarrow \bar X'$ is the open immersion associated with $\bar g'$, we  have a commutative diagram
$$
\xymatrix{
0\ar[r]& \bar\iota_!  \Z_{\bar V} \ar[r] & \Z_{\bar X'} \ar[d]\ar[r] & \bar g'_\ast\Z_{\bar X}\ar[r]\ar[d] & 0 \\
0\ar[r] & g'_\ast \iota_!  \Z_{V} \ar[r] & g'_\ast \Z_{X'} \ar[r] & g'_\ast g_\ast \Z_{X} \ar[r] &0
}
$$
since $\bar g'\circ \bar g=g'\circ g$. Thus, we have a morphism $\bar\iota_!  \Z_{\bar V'} \to g'_\ast \iota_!  \Z_{V'}$ and then a morphism $\bR \bar\phi' \colon (\bar\iota_!  \Z_{\bar V}) \to \bR \bar\phi' (g'_\ast \iota_!  \Z_{V})=\psi'_\ast \bR \phi'_\ast (\iota_!  \Z_{V})$. Taking hypercohomology we get a morphism
$$
T(F',F,\mathbf{\Psi},i) \to T(\bar F',\bar F,\mathbf{\bar \Psi},i)\,.
$$
We have just proved that $T$ associates a morphism to any edge as in (2) of~\ref{sss:type1}. Regarding edges as in (3) of \ref{sss:type1}, consider a chain $F\xrightarrow{\mathbf{\Psi}} F' \xrightarrow{\mathbf{\Psi'}} F''$ of Nori closed embeddings. Denote with primes the morphisms corresponding to $\xi$, $g$ and $\psi$ for $\mathbf{\Psi'}$ and write $\iota\colon V=X'/g(X)\hookrightarrow X'$, $\iota'\colon X''/g'(X') \hookrightarrow X''$ and $\varpi\colon U=X''/(g'\circ g)(X)\hookrightarrow X''$. We have a commutative diagram
$$
\xymatrix@R=10pt{
&& & 0\ar[d] & \\
&& & g'_\ast\iota_!\Z_{V'} \ar[d] & \\
0\ar[r] & \iota'_! \Z_{V'} \ar[r]\ar[d] & \Z_{X''} \ar[r] \ar@{=}[d] & g'_\ast \Z_{X''} \ar[r]\ar[d] & 0 \\
0\ar[r] & \varpi_! \Z_U \ar[r] & \Z_{X''} \ar[r] & g'_\ast g_\ast \Z_X\ar[r]\ar[d] & 0 \\
&& & 0 &
}
$$
so that there is an exact sequence
$0 \to \iota'_! \Z_{V'} \to \varpi_! \Z_U \to g'_\ast\iota_!\Z_{V'} \to 0$. Thus, we have an exact triangle
$$
\bR \phi^{''}_\ast \iota'_! \Z_{V'} \to \bR \phi^{''}_\ast\varpi_! \Z_U \to \bR \phi^{''}_\ast g'_\ast\iota_!\Z_{V'}=\psi_\ast\bR \phi'_\ast \iota_!\Z_{V'} \to \bR \phi^{''}_\ast \iota'_! \Z_{V'}[1]\,,
$$
and then a morphism
\begin{multline*}
T(F',F,i)=\Hs^i(\bR \phi'_\ast \iota_!\Z_{V'})\\
=\Hs^i(\psi_\ast\bR \phi'_\ast \iota_!\Z_{V'}) \to \Hs^{i+1}(\bR \phi^{''}_\ast \iota'_! \Z_{V'}=T(F'',F',i+1)\,,
\end{multline*}
as required.
 \end{example}
 
\begin{example} [Topological effective Type 2   mixed SUSY Nori motives]\label{ex:effNori2}

 A   representation similar to that of the Example \ref{ex:effNori1} can be done in the case of Type 2 SUSY Nori diagrams.

We start with a pre-stable supercurve $(f\colon\Xcal\to\Sc, \{\Xcal_i\},\{\Zc_j\}, \bar\delta)$ and a superscheme morphism $\phi\colon \Xcal \to \Ycal$, where $\Ycal$ is proper and smooth, and $\beta\in A_1(\Ycal)$.
 Given a vertex $(\Sc_1,\Sc_2,i)$, where $\Sc_2\hookrightarrow\Sc_1$ is a closed embedding of $\beta$-good subsuperschemes of $\Sc$, we can consider the induced closed immersion $\delta_{12}\colon \Xcal_{\vert \Sc_1}\hookrightarrow \Xcal_{\vert \Sc_2}$. We then have an exact sequence of sheaves of abelian groups
$$
0 \to \iota_! \Z_{V_{12}}\to \Z_{\Xcal_{\vert \Sc_1}}\to \delta_{12,\ast}\Z_{\Xcal_{\vert \Sc_2}}\to 0
$$
where $\iota\colon V_{12}=\Xcal_{\vert \Sc_1}\setminus \Xcal_{\vert \Sc_2}\hookrightarrow \Xcal_{\vert \Sc_1}$ is the induced open immersion.

We now assign to the vertex $(\Sc_1,\Sc_2,i)$ the abelian group 
$$
T(\Sc_1,\Sc_2,i)= \Hs^i (\bR \phi_\ast (\iota_! \Z_{V_{12}}))\,.
$$
 One can see as in Example \ref{ex:effNori1} that this induces a representation
 $$
 T\colon D_{\mbox{\rm \tiny eff}}(F_{\Sc,f,\phi,\beta}) \to \text{$\Z$-mod}
 $$
 in the sense of \cite[Def. 7.1.4]{HuM-St17}.
\end{example}

 \begin{example} [Geometrical effective Type 2 SUSY mixed SUSY Nori motives.]\label{ex:effNorigeom2}
Take, as in Example \ref{ex:effNori2}, data $F_{\Sc,f,\phi,\beta}$, where $(f\colon\Xcal\to\Sc, \{\Xcal_i\},\{\Zc_j\}, \bar\delta)$ is a pre-stable supercurve,  $\phi\colon \Xcal \to \Ycal$ is a superscheme morphism to a proper and smooth superscheme $\Ycal$, and $\beta\in A_1(\Ycal)$.
 If $\Sc_1\hookrightarrow \Sc$ is a $\beta$-good superscheme, we consider the object of the derived category $D(\Ycal)$ given by
$$
\Tc_1=\bR \phi_{1\ast}(\omega_{f_1})=\bR \phi_{1\ast}(f_1^!\Oc_{\Sc_1}[-1])\,,
$$
where $\phi_1$ and $f_1$ are the restrictions of $\phi$ and $f$ to $\Xcal_1= \Xcal_{\Sc_1}$.
 
 Now, given a vertex $(\Sc_1,\Sc_2,i)$ where $\delta_{12}\colon \Sc_2\hookrightarrow\Sc_1$ is a closed embedding of $\beta$-good subsuperschemes of $\Sc$, applying Proposition \ref{prop:arbitrarybc} on duality base change, one has:
$$
\Tc_2[1]=\bR \phi_{2\ast}(f_2^!\Oc_{\Sc_2})\simeq \bR  \phi_{1\ast}(\delta_{12,\ast}(f_2^!\Oc_{\Sc_2}))\simeq \bR\phi_{1\ast}\delta_{12,\ast} (\bL\delta_{12}^\ast f_1^!\Oc_{\Sc_1})\,,
$$
where we have also denoted by $\delta_{12}$ the induced closed immersion $\Xcal_2\hookrightarrow \Xcal_1$. Then, the natural morphism $f_1^!\Oc_{\Sc_1} \to \delta_{12,\ast} (\bL\delta_{12}^\ast f_1^!\Oc_{\Sc_1})$ induces a morphism $\lambda_{12}\colon \Tc_1\to\Tc_2$ in $D(\Ycal)$.
There is an exact triangle
 $$
 \Tc_1\xrightarrow{\lambda_{12}}\Tc_2 \xrightarrow{\alpha_{12}}\Tc_{12} \xrightarrow{\beta_{12}} \Tc_1[1]\,,
 $$
 where $\Tc_{12}=\operatorname{Cone}(\lambda_{12})$.
 
 We now assign to the vertex  $(\Sc_1,\Sc_2,i)$ the abelian super group
 $$
 T(\Sc_1,\Sc_2,i) = \Hs^i(\Tc_{12})\,.
 $$

One can see straithforwardly that given  two vertices $(\Sc_1,\Sc_2,i)$, $(\Sc'_1,\Sc'_2,i)$ as in (2.a) of Definition \ref{def:Noridiag2}, one has a morphism $ T(\Sc_1,\Sc_2,i)\to T(\Sc'_1,\Sc'_2,i)$.
 
 Regarding (2.b) of Definition \ref{def:Noridiag2}, let us consider  closed immersions  $\delta_{23}\colon \Sc_3\hookrightarrow \Sc_2$, $\delta_{12}\colon \Sc_2\hookrightarrow \Sc_1$ 
 of $\beta$-good subsuperschemes of $\Sc$, and let $\delta_{13}\colon \Sc_3\hookrightarrow \Sc_1$ be the composition. We then have morphisms
 $$
 \Tc_{23}\xrightarrow{\beta_{23}}\Tc_2[1]\xrightarrow{\alpha_{12}[1]}\Tc_{12}[1]\,,
 $$
 and hence a supergroup morphism
 $$
 T(\Sc_2,\Sc_3,i) \to T(\Sc_1,\Sc_2,i+1) 
 $$
 as required. 
Then $T$ is a representation 
 $$
 T\colon D_{\mbox{\rm \tiny eff}}(F_{\Sc,f,\phi,\beta}) \to \text{Abelian supergroups}=\text{$\Z_2$-graded $\Z$-mod}
 \,.$$
\end{example}

\def\cprime{$'$}

\bigskip

\address{Departamento de Matem\'atica, Instituto de Ci\^encias Exatas,
Universidade Federal de Minas Gerais, Av.~Ant\^onio Carlos 6627, Pampulha,
Belo Horizonte, 31270-901 MG, Brazil;
IGAP (Institute for Geometry and Physics), Trieste, Italy \\
\email{\indent ubruzzo@ufmg.br}}

\address{Departamento de Matem\'aticas and IUFFYM (Instituto Universitario de F\'{\i}sica Fundamental y Matem\'a\-ticas), Universidad de Salamanca, Plaza de la Merced 1-4, 37008 Salamanca, Spain; 
Real Academia de Ciencias Exactas, F\'{i}sicas y Naturales, Spain\\
\email{\indent ruiperez@usal.es}}

\address{Max-Planck-Institut f\"ur Mathematik, 
	Vivatsgasse 7, 53111 Bonn, Germany\\

\end{document}